\documentclass{article}

\usepackage
{
    graphicx,
    amssymb,
    amsmath,
    amsthm,
    xcolor,
    epstopdf,
    algorithm,
    algpseudocode,
    authblk,
    mathtools
}

\usepackage[colorlinks,
            pdffitwindow=false,
            plainpages=false,
            pdfpagelabels=true,
            pdfpagemode=UseOutlines,
            pdfpagelayout=SinglePage,
            bookmarks=false,
            colorlinks=true,
            hyperfootnotes=false,
            linkcolor=blue,
            citecolor=green!50!black]{hyperref}

\usepackage[hmargin=2cm,vmargin=2.5cm]{geometry}
\usepackage[bf]{caption}

\DeclareMathAlphabet{\mathpzc}{OT1}{pzc}{m}{it}

\newcommand{\subfiguretitle}[1]{{\scriptsize{#1}} \\[1mm]}
\newcommand{\R}{\mathbb{R}}                                     % real numbers
\providecommand{\abs}[1]{\left\lvert #1 \right\rvert}           % absolute value
\providecommand{\norm}[1]{\left\lVert #1 \right\rVert}          % norm

\newcommand{\mat}[3]{#1 \left| \begin{array}{@{}l@{}} \scriptstyle #3 \\ \scriptstyle #2 \end{array}  \right.}

\newcommand\xqed[1]{\leavevmode\unskip\penalty9999 \hbox{}\nobreak\hfill \quad\hbox{#1}}
\newcommand{\exampleSymbol}{\xqed{$\triangle$}}

\DeclareMathOperator{\diag}{diag}

\newtheorem{theorem}{Theorem}[section]

\newtheorem{lemma}[theorem]{Lemma}
\newtheorem{proposition}[theorem]{Proposition}
\newtheorem{definition}[theorem]{Definition}
\theoremstyle{definition}
\newtheorem{example}[theorem]{Example}
\newtheorem{remark}[theorem]{Remark}

\title{Tensor-based dynamic mode decomposition}
\author[1]{Stefan Klus}
\author[1]{Patrick Gel\ss}
\author[2]{Sebastian Peitz}
\author[1,3]{Christof Sch\"utte}
\affil[1]{\normalsize Department of Mathematics and Computer Science, Freie Universit\"at Berlin, Germany}
\affil[2]{\normalsize Department of Mathematics, University of Paderborn, Germany}
\affil[3]{\normalsize Zuse Institute Berlin, Germany}
\date{}

\begin{document}
\maketitle

\begin{abstract}
Dynamic mode decomposition (DMD) is a recently developed tool for the analysis of the behavior of complex dynamical systems. In this paper, we will propose an extension of DMD that exploits low-rank tensor decompositions of potentially high-dimensional data sets to compute the corresponding DMD modes and eigenvalues. The goal is to reduce the computational complexity and also the amount of memory required to store the data in order to mitigate the curse of dimensionality. The efficiency of these tensor-based methods will be illustrated with the aid of several different fluid dynamics problems such as the von K\'arm\'an vortex street and the simulation of two merging vortices.
\end{abstract}

\section{Introduction}

Dynamic mode decomposition, which was first introduced in~\cite{SS08}, is a powerful tool for analyzing the behavior of complex dynamical systems and can, for instance, be used to identify low-order dynamics~\cite{TRLBK14}. Over the last years, several variants such as exact, optimized, or sparsity-promoting DMD have been proposed~\cite{SS08, TRLBK14, CTR12, JSN14}. It was also shown that DMD is closely related to the Koopman operator analysis. A generalization of DMD called extended dynamic mode decomposition (EDMD) is presented in \cite{WKR15} and has been developed for the approximation of the Koopman operator and its eigenvalues, eigenfunctions, and eigenmodes. In the same way, DMD and EDMD can be used to approximate the Perron--Frobenius operator -- the adjoint of the Koopman operator -- as shown in~\cite{KKS16}.

The amount of data that can be analyzed using methods like DMD or EDMD is limited. Due to the so-called \emph{curse of dimensionality}, analyzing high-dimensional problems becomes infeasible. This can be mitigated by exploiting low-rank tensor approximation approaches. Several tensor formats such as the canonical tensor format, the Tucker format, or the tensor-train format (TT-format) have been proposed. While the canonical format would be optimal from a simplicity and efficiency point of view, it was shown to be numerically unstable \cite{dSL08, HRS12}. Our method relies on the tensor-train decomposition~\cite{Ose11}, which can be regarded as a multi-dimensional generalization of the conventional singular value decomposition (SVD).

The goal of this paper is to extend DMD to use tensors instead of vectors so that low-rank approximations of the data can be utilized in order to reduce the computational complexity and also the amount of memory required to store the data and the resulting linear operators and their eigenfunctions, eigenvalues, and corresponding modes. To this end, we will have to compute singular value decompositions and pseudoinverses of tensor unfoldings. Recently, a method to compute an approximation of the pseudoinverse of operators in tensor-train format, which relies on the solution of an optimization problem, was proposed in~\cite{LC16}. The optimization problem is solved with the Modified Alternating Linear Scheme (MALS) \cite{HRS12}, breaking the problem into smaller subproblems which can be solved with conventional methods. We will use a different approach here. Since the pseudoinverse of a matrix can be computed based on a singular value decomposition, the question now is whether the TT-representation of a tensor itself already contains information about the pseudoinverse that can be exploited without necessitating the solution of an optimization problem.

The outline of this paper is as follows: Chapter \ref{sec:Dynamic mode decomposition} briefly recapitulates the properties of the pseudoinverse of a matrix and describes standard and exact DMD. Chapter~\ref{sec:Tensor-based dynamic mode decomposition} introduces tensors and different tensor representations, in particular the TT-format. Furthermore, we will show how the pseudoinverse required for DMD can be obtained from a given tensor decomposition of the data. Then we will present a reformulation of DMD using tensors instead of matrices and vectors. In Chapter~\ref{sec:Numerical results}, numerical results for different fluid dynamics problems will be presented. Chapter~\ref{sec:Conclusion} concludes with a brief summary, open problems, and future work.

\section{Dynamic mode decomposition}
\label{sec:Dynamic mode decomposition}

In this section, we will introduce DMD and show how the DMD modes and eigenvalues can be computed efficiently using a reduced singular value decomposition of the data. The goal is then to rewrite these methods in terms of tensors, exploiting low-rank decompositions of potentially high-dimensional data. These tensor decomposition techniques will be introduced in Section~\ref{sec:Tensor-based dynamic mode decomposition}.

DMD decomposes high-dimensional data into coupled spatial-temporal modes and can be regarded as a combination of a principal component analysis (PCA) in the spatial domain and a Fourier analysis in the frequency domain~\cite{BJOK14}. The DMD modes often correspond to coherent structures in the flow, the DMD eigenvalues can be interpreted as growth rates and frequencies of the corresponding modes. Thus, it is often used to analyze oscillatory behavior~\cite{TRLBK14}. Originally applied to time series data, DMD has been generalized to analyze pairs of $ n $-dimensional data vectors $ x_i $ and $ y_i = F(x_i) $, $ i = 1, \dots, m $, written in matrix form as
\begin{equation} \label{eq:DMD data}
    X =
    \begin{bmatrix}
        x_1 & x_2 & \cdots & x_m
    \end{bmatrix}
    \quad \text{and} \quad
    Y =
    \begin{bmatrix}
        y_1 & y_2 & \cdots & y_m
    \end{bmatrix}.
\end{equation}
Here, $ F : \R^n \to \R^n $ can be any (non-)linear dynamical system. Thus, for a given time series $ \{ z_0, \, z_1, \, \dots, \, z_m\} $, that is $ z_i = F^i(z_0) $, we obtain
\begin{equation}
    X =
    \begin{bmatrix}
        z_0 & z_1 & \cdots & z_{m-1}
    \end{bmatrix}
    \quad \text{and} \quad
    Y =
    \begin{bmatrix}
        z_1 & z_2 & \cdots & z_m
    \end{bmatrix}.
\end{equation}
Assuming there exists a linear operator $ A $ that describes the dynamics of the system such that $ y_i = A x_i $, define
\begin{equation} \label{eq:DMD A}
    A = Y X^+,
\end{equation}
where $ ^+ $ denotes the pseudoinverse. The matrix $ A $ minimizes the cost function $ \norm{A X - Y}_F $, where $ \norm{.}_F $ denotes the Frobenius norm. The DMD modes and eigenvalues are then defined to be the eigenvectors and eigenvalues of~$ A $. Using $ X^+ = X^T (X X^T)^+ $, we obtain
\begin{equation}
    A = (Y X^T) (X X^T)^+.
\end{equation}
The difference is that now the pseudoinverse of an $ (n \times n) $-matrix needs to be computed and not the pseudoinverse of an $ (n \times m) $-matrix, which is advantageous if $ n \ll m $. In our fluid dynamics examples, however, the dimension $ n $ of the problem is typically much larger than the number of snapshots $ m $.

\subsection{Singular value decomposition and the pseudoinverse}

Before we begin with the conventional formulation of the DMD algorithm, let us briefly recapitulate the properties of the standard singular value decomposition and the standard pseudoinverse of a matrix. These properties and definitions will be used and generalized in Section~\ref{sec:Tensor-based dynamic mode decomposition} for the computation of singular value decompositions and pseudoinverses of tensor unfoldings.

Let $ M \in \R^{n_1 \times n_2} $ be a matrix. It is well known that the pseudoinverse $ M^+ \in \R^{n_2 \times n_1} $ can be computed by a (compact/reduced) singular value decomposition of $ M $. Assume that $ M = U \, \Sigma \, V^T $, where $ U \in \R^{n_1 \times s} $, $ V \in \R^{n_2 \times s} $, and $ \Sigma = \diag(\sigma_1, \dots, \sigma_s) \in \R^{s \times s} $ is a diagonal matrix containing only the nonzero singular values. Then the pseudoinverse is given by
\begin{equation*}
    M^+ = V \, \Sigma^{-1}  \, U^T,
\end{equation*}
where $ \Sigma^{-1} = \diag(\sigma_1^{-1}, \dots, \sigma_s^{-1}) $. In what follows, we will assume that the singular values are sorted in decreasing order, i.e.~$\sigma_1 \ge \sigma_2 \ge \dots \ge \sigma_s > 0 $. Let us now recall the definition of the tensor product. For two vectors $v \in \R^{n_1}$ and $w \in \R^{n_2}$, the tensor product $ v \otimes w \in \R^{n_1 \times n_2} $ is given by
\begin{equation} \label{eq: outer tensor product}
    (v \otimes w)_{i,j} = (v \cdot w^T)_{i,j} = v_i \cdot w_j.
\end{equation}

\begin{example}
For a rank-$ s $ matrix
\begin{equation*}
    M = U \, \Sigma \, V^T
      = \sum_{i=1}^s \sigma_i u_i v_i^T
      = \sum_{i=1}^s \sigma_i u_i \otimes v_i,
\end{equation*}
we would obtain
\begin{equation*}
    M^+ = V \, \Sigma^{-1} \, U^T
        = \sum_{i=1}^s \frac{1}{\sigma_i} v_i u_i^T
        = \sum_{i=1}^s \frac{1}{\sigma_i} v_i \otimes u_i. \tag*{\exampleSymbol}
\end{equation*}
\end{example}

The example shows that in order to obtain the pseudoinverse we simply exchange the roles of the left and right singular vectors and divide each tensor product $ v_i \otimes u_i $ by the corresponding singular value $ \sigma_i $. These properties will be used later on to construct the pseudoinverse of a tensor given in TT-format, which is based on successive singular value decompositions.

\subsection{Computation of DMD modes and eigenvalues}

There are different algorithms to compute the DMD modes and eigenvalues without explicitly computing $ A $ which rely on a compact singular value decomposition of $ X $. The standard DMD version is shown in Algorithm~\ref{alg:StandardDMD}.

\begin{algorithm}[htb]
    \caption{Standard DMD algorithm.}
    \label{alg:StandardDMD}
    \begin{algorithmic}[1]
        \State Compute the compact SVD of $ X $, given by $ X = U \, \Sigma \, V^T $ with $ U \in \R^{n \times s} $, $ V \in \R^{m \times s} $, and $ \Sigma \in \R^{s \times s} $.
        \State Define $ \tilde{A} = U^T Y V \Sigma^{-1} $.
        \State Compute eigenvalues and eigenvectors of $ \tilde{A} $, i.e.~$ \tilde{A} w = \lambda w $.
        \State The DMD mode corresponding to the eigenvalue $ \lambda $ is then defined as $ \varphi = U w $.
    \end{algorithmic}
\end{algorithm}

Algorithm~\ref{alg:ExactDMD} shows a variant called exact DMD. The modes computed by the standard DMD algorithm are simply the modes computed by the exact DMD algorithm projected onto the range of $ X $. For a more detailed description, we refer to~\cite{TRLBK14}.

\begin{algorithm}[htb]
    \caption{Exact DMD algorithm.}
    \label{alg:ExactDMD}
    \begin{algorithmic}[1]
        \State Execute steps 1 to 3 of Algorithm~\ref{alg:StandardDMD}.
        \State The DMD mode corresponding to the eigenvalue $ \lambda $ is then defined as $ \varphi = \frac{1}{\lambda} Y V \Sigma^{-1} w $.
    \end{algorithmic}
\end{algorithm}

The matrix $A$ defined in~\eqref{eq:DMD A} and the matrix $\tilde{A}$ given in Algorithm~\ref{alg:StandardDMD} share the same eigenvalue spectrum. However, the DMD algorithm computes only the nonzero eigenvalues. If $\lambda$ is an eigenvalue of $A$ corresponding to the eigenvector $\varphi$, i.e.~$A \varphi = \lambda \varphi$, then it follows that $\tilde{A}w = \lambda w$ with $w = U^T \varphi$. Conversely, if we have $\tilde{A} w = \lambda w$ and define $\varphi = \frac{1}{\lambda} Y V \Sigma^{-1} w $ as in Algorithm \ref{alg:ExactDMD}, then $A v = \lambda v$. A proof of this result can also be found in~\cite{TRLBK14}. In order to illustrate the idea behind DMD, let us analyze a simple example.

\begin{example} \label{ex:Karman 2D}
Consider the von K\'arm\'an vortex street. The system describes the flow past a two-dimensional cylinder which -- for a certain range of Reynolds numbers -- results in a repeating pattern of vortices. A similar example and more details about the characteristic properties of the system can be found in~\cite{TRLBK14}. Snapshots of the solution of this partial differential equation at intermediate time steps are shown in Figure~\ref{fig:Karman}a. We discretized the domain using a $ 60 \times 120 $ grid and generated $ 101 $ snapshots of the solution at equidistant time points. The corresponding DMD modes are shown in Figure~\ref{fig:Karman}b. \exampleSymbol

\begin{figure}[htb]
    \centering
    \subfiguretitle{a)}
    \includegraphics[width=0.95\textwidth]{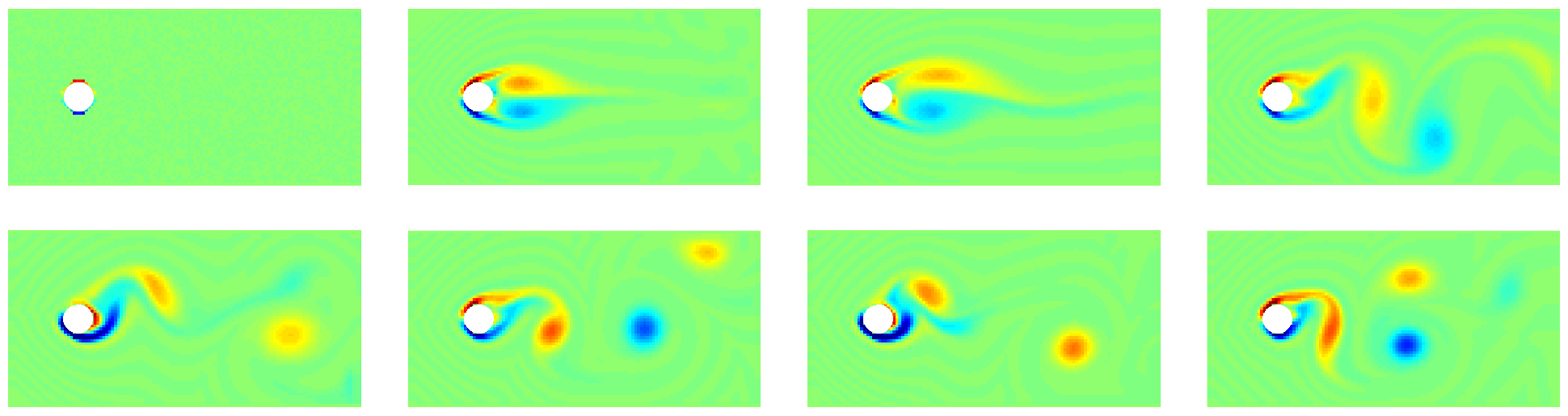} \\
    \subfiguretitle{b)}
    \includegraphics[width=0.95\textwidth]{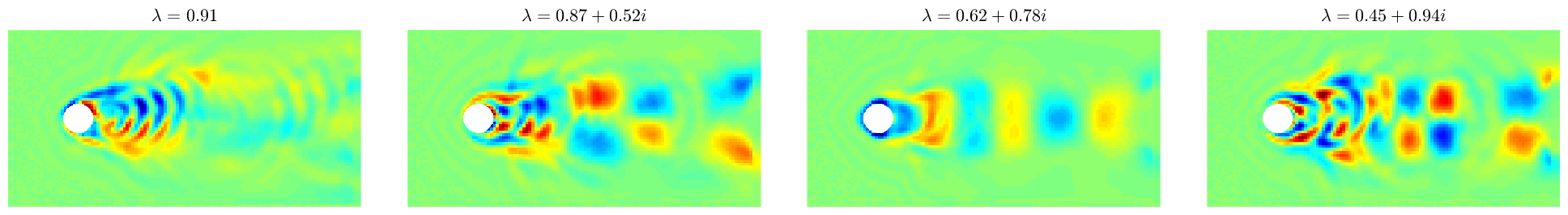}
    \caption{a) Simulation of the von K\'arm\'an vortex street. b) DMD modes corresponding to eigenvalues close to 1.}
    \label{fig:Karman}
\end{figure}

\end{example}

Note that for this example we reshaped each $ 60 \times 120 $ snapshot matrix as a vector, resulting in data matrices $ X, Y \in \R^{7200 \times 100} $. Our reformulation of DMD, on the other hand, will directly operate on tensors $ \mathbf{X}, \mathbf{Y} \in \R^{60 \times 120 \times 100} $.

\section{Tensor-based dynamic mode decomposition}
\label{sec:Tensor-based dynamic mode decomposition}

In this section, we will first introduce different tensor formats and then show how the pseudoinverse of a tensor unfolding can be computed efficiently exploiting properties of the TT-decomposition. Furthermore, we will rewrite DMD in terms of tensors. The aim is to exploit low-rank tensor approximation methods to analyze high-dimensional dynamical systems that could not be handled before due to the curse of dimensionality.

\subsection{The TT-format}

Over the last years, different tensor formats such as the canonical format, the Tucker format, and the tensor-train format have been developed, see e.g.~\cite{BM02, GKT13, Hac12, Hac14}. A frequently used format is the TT-format~\cite{Ose11}. A tensor $ \mathbf{x} \in \R^{n_1 \times \dots \times n_d} $ is decomposed into $ d $ component tensors $ \mathbf{x}^{(i)} $ of at most order three, where the first and last component tensors of order two are, for simplicity, often interpreted as tensors of order three with the additional requirement that $ r_0 = r_d = 1 $. The entries of $ \mathbf{x} $ are given by
\begin{equation}
    \begin{split}
        \mathbf{x}_{i_1, \dots, i_d}
            &= \sum_{k_0=1}^{r_0} \dots \sum_{k_{d}=1}^{r_{d}}
            \mathbf{x}^{(1)}_{k_0, i_1, k_1} \cdot \mathbf{x}^{(2)}_{k_1, i_2, k_2} \cdot \dotsc \cdot \mathbf{x}^{(d-1)}_{k_{d-2}, i_{d-1}, k_{d-1}} \cdot \mathbf{x}^{(d)}_{k_{d-1}, i_d, k_d} \\
            &= \mathbf{x}^{(1)}_{:, i_1, :} \cdot \mathbf{x}^{(2)}_{:, i_2, :} \cdot \dotsc \cdot \mathbf{x}^{(d-1)}_{:, i_{d-1}, :} \cdot \mathbf{x}^{(d)}_{:, i_d,:},
    \end{split}
\end{equation}
where the more compact second formulation uses Matlab's colon notation. With the aid of the tensor product \eqref{eq: outer tensor product}, the whole tensor can then be represented as
\begin{equation}
    \mathbf{x} = \sum_{k_0=1}^{r_0} \dots \sum_{k_d=1}^{r_d} \mathbf{x}^{(1)}_{k_0,:,k_1} \otimes \mathbf{x}^{(2)}_{k_1,:,k_2} \otimes \dots \otimes \mathbf{x}^{(d-1)}_{k_{d-2},:,k_{d-1}} \otimes \mathbf{x}^{(d)}_{k_{d-1},:,k_d}.
\end{equation}
The vector $ r = [r_0, \dots, r_d] $ contains the ranks of the TT-tensor and determines the complexity of the representation. The lower the ranks, the lower the memory consumption and the computational costs. One of the main advantages of the TT-format, compared to the canonical format, is its stability from an algorithmic point of view~\cite{dSL08, HRS12}.

For computational aspects, it is helpful to reshape tensors into matrices and vectors. In order to describe matricizations and vectorizations -- also called \emph{tensor unfoldings} --, we first define a bijection $\phi_N$ for the ordered set $N = (n_1, \ldots, n_d)$ by
\begin{equation} \label{bijection_1}
    \begin{gathered}
        \phi_N : \{1, \ldots, n_1 \} \times \ldots \times \{1, \ldots, n_d \} \rightarrow \{1, \ldots, \prod_{\mu=1}^{d} n_\mu\}, \\
        (i_1, \ldots, i_d) \mapsto \phi_N(i_1, \ldots, i_d).
    \end{gathered}
\end{equation} 
Using the \emph{little-endian} convention, this bijection is defined as
\begin{equation} \label{eq_phi_formula}
    \phi_N(i_1, \ldots, i_d) = 1+ (i_1 -1) + \ldots + (i_d -1) \cdot n_1 \cdot \ldots \cdot n_{d-1} = 1 + \sum_{\mu=1}^{d} (i_\mu -1) \prod_{\nu=1}^{\mu-1} n_\nu.
\end{equation} 
If the definition of $N$ is clear from the context, we write $\phi_N(i_1, \ldots , i_d) = \overline{i_1 , \ldots , i_d}$.
\begin{definition}
Let  $\mathbf{x} \in \mathbb{R}^{n_1 \times \dots \times n_d}$ be a tensor. For the two ordered subsets $N' = (n_{1}, \ldots, n_{l})$ and $N'' = (n_{l+1}, \ldots, n_{d})$ of $ N $, $ 1 \leq l \leq d-1 $, the matricization $\mat{\mathbf{x}}{N'}{N''} \in \R^{(n_1 \cdot \dotsc \cdot n_l) \times (n_{l+1} \cdot \dotsc \cdot n_d)} $ of $\mathbf{x}$ with respect to $ N' $ and $ N''$ is given by
\begin{equation} \label{eq_matricization}
    \left( \mat{\mathbf{x}}{N'}{N''}\right)_{\phi_{N'} (i_{1}, \ldots, i_{l}), \phi_{N''}(i_{l+1}, \ldots, i_{d})} = \left( \mat{\mathbf{x}}{N'}{N''}\right)_{\overline{i_{1}, \ldots, i_{l}}, \overline{i_{l+1}, \ldots, i_{d}}} = \mathbf{x}_{i_1, \ldots, i_d}.
\end{equation}
Accordingly, the vectorization $\mat{\mathbf{x}}{N}{~} \in \R^{n_1 \cdot \dotsc \cdot n_d}$ of a tensor $\mathbf{x} \in \mathbb{R}^{n_1 \times \dotsc \times n_d}$ is given by a matricization of $\mathbf{x}$ with $N' = N$ and $N'' = { \varnothing } $, i.e.
\begin{equation} \label{eq:vectorization}
    \left( \mat{\mathbf{x}}{N}{~}\right)_{\phi_{N}(i_{1}, \ldots, i_{d})} = \left( \mat{\mathbf{x}}{N}{~}\right)_{\overline{i_{1}, \ldots, i_{d}}} = \mathbf{x}_{i_1, \ldots, i_d}.
\end{equation}
\end{definition}

\begin{remark}
Note that here and in what follows vectors $\mathbf{x}^{(\mu)}_{k_{\mu-1}, :, k_\mu}$ as well as vectorizations $ \mat{\mathbf{x}}{N}{~} $ of tensors are regarded as column vectors.
\end{remark}

If a tensor $\mathbf{x}$ is given in full format, i.e.~$\mathbf{x}$ is indeed a $ d $-dimensional array, Algorithm~\ref{alg:TT-decomposition} can be used to compute an exact ($\varepsilon = 0$) or an approximated ($\varepsilon >0$) TT-decomposition of $\mathbf{x}$, respectively. For more details, see~\cite{Ose11}. Here, $ U_{\mu} $ and $ V_{\mu} $ denote the matrices containing the first $ {r_\mu} $ singular vectors and $ \Sigma_{\mu} $ denotes the diagonal matrix comprising the first $ {r_\mu} $ singular values.

\begin{algorithm}[htbp]
    \caption{Convert a tensor given in full format into the TT-format.}
    \label{alg:TT-decomposition}
    \begin{algorithmic}[1]
    \State Given a tensor $ \mathbf{x} \in \R^{n_1 \times \cdots \times n_d} $ in full format and a threshold $ \varepsilon $.
    \For{$ \mu =  1, ..., d-1 $}
        \State $ M = \mat{\mathbf{x}}{r_{\mu-1}, n_{\mu}}{n_{\mu+1}, \dotsc, n_d} $.
        \State Compute SVD of $ M $, i.e.~$ M = U_\mu \Sigma_\mu V_\mu^T $ with $\Sigma_\mu \in \R^{s \times s}$.
        \State Set $ r_\mu \leq s $ to the largest index such that $ (\Sigma_\mu)_{i,i} > \varepsilon $ for $i \leq r_{\mu}$.
        \State Discard rows and columns of $U_\mu$, $\Sigma_\mu$, and $V_\mu$ corresponding to singular values smaller than or equal to $ \varepsilon $.
        \State Set $ \mathbf{y}^{(\mu)} \in \R^{r_{\mu -1} \times  n_{\mu} \times r_{\mu}}$ to a reshaped version of $U_\mu$ with $ \mathbf{y}^{(\mu)}_{k_{\mu-1}, i_{\mu}, k_{\mu}} = (U_\mu)_{\overline{k_{\mu-1}, i_{\mu}}, k_{\mu}}$.
        \State Define remainder $ \mathbf{x} = \Sigma_\mu \, V_\mu^T \in \R^{r_\mu \times n_{\mu +1} \cdot \dotsc \cdot n_d} $.
    \EndFor
    \State Set $ d $-th core to $ \mathbf{y}^{(d)}_{:,:,1} = \mathbf{x} $.
    \State The tensor $ \mathbf{y} $ with cores $ \mathbf{y}^{(1)}, \dots, \mathbf{y}^{(d)} $ and ranks $ r_0, \dots, r_d $ is then an approximation of the initial tensor $ \mathbf{x} $.
    \end{algorithmic}
\end{algorithm}

Orthonormality of tensor trains plays an important role, in particular when we want to compute pseudoinverses of a tensor given in TT-format.

\begin{definition}
A TT-core $\mathbf{x}^{(\mu)} \in \mathbb{R}^{r_{\mu-1} \times n_\mu \times r_\mu}$ is called \emph{left-orthonormal} if 
\begin{equation} \label{eq:left-orth}
    \left( \mat{\mathbf{x}^{(\mu)}}{r_{\mu-1}, n_\mu}{r_\mu} \right)^T \cdot 
    \left( \mat{\mathbf{x}^{(\mu)}}{r_{\mu-1}, n_\mu}{r_\mu} \right) = I \in \mathbb{R}^{r_\mu \times r_\mu}.
\end{equation}
Correspondingly, $\mathbf{x}^{(\mu)}$ is called \emph{right-orthonormal} if 
\begin{equation} \label{eq:right-ortho}
    \left( \mat{\mathbf{x}^{(\mu)}}{r_{\mu-1}}{n_\mu, r_\mu} \right) \cdot
    \left( \mat{\mathbf{x}^{(\mu)}}{r_{\mu-1}}{n_\mu, r_\mu} \right)^T  = I \in \mathbb{R}^{r_{\mu-1} \times r_{\mu-1}}.
\end{equation}
\end{definition}

Naturally, the first $d-1$ TT-cores of $\mathbf{y}$ as constructed in Algorithm \ref{alg:TT-decomposition} are left-orthonormal while the last core is right-orthogonal, but not right-orthonormal, see the following Lemma.

\begin{lemma} \label{lem:TT properties}
Due to the construction of the TT-representation of a tensor as described in Algorithm~\ref{alg:TT-decomposition}, it holds that:
\begin{enumerate}
\item $ \mat{\mathbf{y}^{(\mu)}}{r_{\mu-1}, n_\mu}{r_\mu} $ is left-orthonormal for $\mu = 1, \dotsc, d-1$.
\item $ \left( \mat{\mathbf{y}^{(d)}}{r_{d-1}}{n_d, r_d}\right) \cdot  \left( \mat{\mathbf{y}^{(d)}}{r_{d-1}}{n_d, r_d}\right)^T = \Sigma_{d-1}^2$, where $\Sigma_{d-1}$ is part of the last SVD in Algorithm~\ref{alg:TT-decomposition}.
\end{enumerate}
\end{lemma}
\begin{proof}
The properties are a result of the fact that the cores $ \mathbf{y}^{(\mu)} $, $\mu = 1, \dotsc, d-1$, are reshaped versions of matrices $U_\mu$ with $U_\mu^T \cdot U_\mu = I$ and $ \mathbf{y}^{(d)} = \Sigma_{d-1} \, V_{d-1}^T $ with $V_{d-1}^T \cdot V_{d-1} = I$.
\end{proof}

In general, we do, however, not want to compute the tensor in full format and then convert it to the TT-format. All the numerical computations should ideally be directly carried out in the TT-format. This could, for instance, mean solving systems of equations, eigenvalue problems, ordinary or partial differential equations, or completion problems using the TT-format, see e.g.~\cite{BG13,DKOS14,GMS16,Dol14,RSS14}. In this way, we automatically compute a low-rank approximation of the solution without necessitating the conversion to the TT-format. The aim is now to exploit properties of the TT-decomposition in order to efficiently compute DMD modes and eigenvalues. For this purpose, we will need the following auxiliary results:

\begin{lemma} \label{lem:Tensor product properties}
The tensor product satisfies
\begin{flalign*}
    && \left( \mat{\left( \mathbf{x}^{(1)} \otimes \dotsc \otimes \mathbf{x}^{(d)} \right)}{n_1, \dotsc, n_d}{~}\right)^T \cdot \left( \mat{\left( \mathbf{y}^{(1)} \otimes \dotsc \otimes \mathbf{y}^{(d)} \right)}{n_1, \dotsc, n_d}{~}\right) &= \displaystyle \prod_{\mu=1}^d \left( \mathbf{x}^{(\mu)}\right)^T \cdot \mathbf{y}^{(\mu)} & \\
    \text{and} &&
    \left\| \mat{\left( \mathbf{x}^{(1)} \otimes \dotsc \otimes \mathbf{x}^{(d)} \right)}{n_1, \dotsc, n_d}{~} \right\|_2 &= \displaystyle \prod_{\mu=1}^d \left\| \mathbf{x}^{(\mu)} \right\|_2.
\end{flalign*}

\end{lemma}
\begin{proof}
The first equation follows from
\begin{equation*}
    \begin{split}
        & \left( \mat{\left( \mathbf{x}^{(1)} \otimes \dotsc \otimes \mathbf{x}^{(d)} \right)}{n_1, \dotsc, n_d}{~}\right)^T \cdot \left( \mat{\left( \mathbf{y}^{(1)} \otimes \dotsc \otimes \mathbf{y}^{(d)} \right)}{n_1, \dotsc, n_d}{~}\right) \\
        & \quad = \sum_{i_1 = 1}^{n_1}  \cdots \sum_{i_d = 1}^{n_d} \mathbf{x}^{(1)}_{i_1} \cdot \dotsc \cdot \mathbf{x}^{(d)}_{i_d} \cdot \mathbf{y}^{(1)}_{i_1} \cdot \dotsc \cdot \mathbf{y}^{(d)}_{i_d}
        = \prod_{\mu=1}^d \left( \mathbf{x}^{(\mu)}\right)^T \cdot \mathbf{y}^{(\mu)},
    \end{split}
\end{equation*}
the second is simply a special case of the first.
\end{proof}

\subsection{Singular value decomposition and the pseudoinverse in TT-format}

Before we consider TT-tensors with arbitrary ranks, let us illustrate the basic idea with a simple example.

\begin{example} \label{ex:rank-1 pseudoinverse}
Assume that we have a rank-one tensor of the form $ \mathbf{x} = \mathbf{x}^{(1)} \otimes \mathbf{x}^{(2)} \otimes \mathbf{x}^{(3)} \otimes \mathbf{x}^{(4)} \in \R^{n_1 \times n_2 \times n_3 \times n_4} $, with $ \mathbf{x}^{(i)} \in \R^{n_i} $, and we want to compute the pseudoinverse of the matricization with respect to the dimensions $ (1, 2) $ and $ (3, 4) $, given by
\begin{equation*}
    M = \mat{\left(\mathbf{x}^{(1)} \otimes \mathbf{x}^{(2)}\right)}{n_1, n_2}{~} \otimes \mat{\left(\mathbf{x}^{(3)} \otimes \mathbf{x}^{(4)}\right)}{n_3, n_4}{~}
    \in \R^{(n_1 \cdot n_2) \times (n_3 \cdot n_4)}.
\end{equation*}
Then
\begin{equation*}
    M^+ = \frac{1}{ \sigma^2 }  \mat{\left(\mathbf{x}^{(3)} \otimes \mathbf{x}^{(4)}\right)}{n_3, n_4}{~} \otimes \mat{\left(\mathbf{x}^{(1)} \otimes \mathbf{x}^{(2)}\right)}{n_1, n_2}{~}
    \in \R^{(n_3 \cdot n_4) \times (n_1 \cdot n_2)},
\end{equation*}
with $ \sigma = \prod_{i=1}^4 \norm{\mathbf{x}^{(i)}}_2 $. This can be seen by writing
\begin{equation*}
    M = \underbrace{\mat{\left(\mathbf{x}^{(1)} \otimes \mathbf{x}^{(2)}\right)}{n_1, n_2}{~}}_{\tilde{u}} \otimes \underbrace{\mat{\left(\mathbf{x}^{(3)} \otimes \mathbf{x}^{(4)}\right)}{n_3, n_4}{~}}_{\tilde{v}}
        = \tilde{u} \otimes \tilde{v}
        = \sigma \underbrace{\left( \frac{1}{\norm{\tilde{u}}} \tilde{u} \right)}_{u} \otimes \underbrace{\left( \frac{1}{\norm{\tilde{v}}} \tilde{v} \right)}_{v}
\end{equation*}
and thus
\begin{equation*}
    M^+ = \frac{1}{\sigma} v \otimes u
        = \frac{1}{\sigma^2} \tilde{v} \otimes \tilde{u}.
\end{equation*}
The formula for $ \sigma $ follows from Lemma~\ref{lem:Tensor product properties}. With a slight abuse of notation, we will write
\begin{equation*} \label{eq:Pseudoinverse of rank-1 tensor}
    \mathbf{x}^+ = \frac{1}{ \sigma^2 } \mathbf{x}^{(3)} \otimes \mathbf{x}^{(4)} \otimes \mathbf{x}^{(1)} \otimes \mathbf{x}^{(2)}.
\end{equation*}
That is, the pseudoinverse can in this case simply be obtained by reordering the cores of the rank-$ 1 $ tensor and by normalizing the tensor product. Note that if the tensor was obtained by the TT-decomposition, then the vectors $ \mathbf{x}^{(1)} $, $ \mathbf{x}^{(2)} $, and $ \mathbf{x}^{(3)} $ are already normalized and we have to divide only by $ \norm{\mathbf{x}^{(4)}}^2 $. A graphical representation of this process is shown in Figure~\ref{fig:Pseudoinverse rank 1}, where we use a similar diagrammatic notation as in \cite{HRS12}. \exampleSymbol

\begin{figure}[htb]
    \centering
    \includegraphics[width=0.7\textwidth]{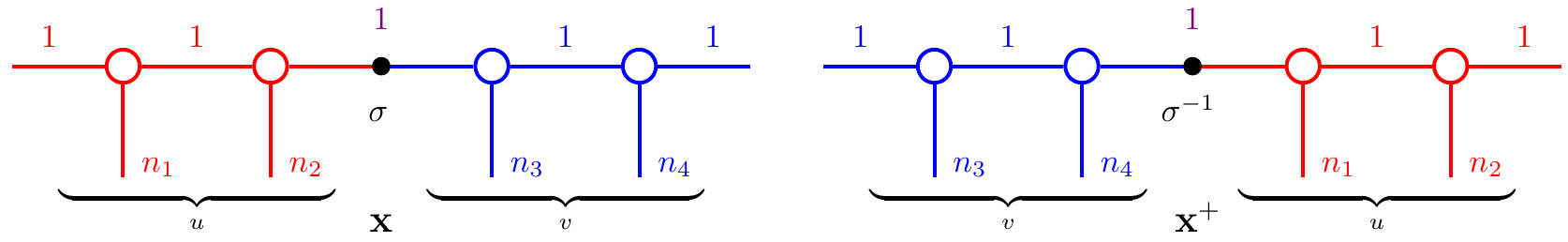}
    \caption{Pseudoinverse of a rank-$ 1 $ tensor with respect to the dimensions $ (1, 2) $ and $ (3, 4) $. The first two and the last two cores are swapped and the tensor is divided by $ \sigma $.}
    \label{fig:Pseudoinverse rank 1}
\end{figure}

\end{example}

Note that the pseudoinverse depends on the matricization of the tensor $ \mathbf{x} $. For a different matricization, for instance with respect to the dimensions $ (1, 2, 3) $ and $ (4) $, we would also obtain a different pseudoinverse $ \mathbf{x}^+ $. We now want to generalize this approach to compute pseudoinverses of arbitrary tensors $ \mathbf{x} $ in TT-format for a given matricization. It turns out that also for tensors with a higher-rank coupling the pseudoinverse can be obtained by reordering the cores after some preprocessing steps. These include the left- and right-orthonormalization, respectively, of the given TT-cores.

\begin{algorithm}[htbp]
    \caption{Left-orthonormalize TT-cores $\mathbf{x}^{(1)}, \dotsc, \mathbf{x}^{(l)}$.}
    \label{alg:left-ortho}
    \begin{algorithmic}[1]
        \State Given tensor $ \mathbf{x} $ and core number $ l $, $1 \leq l \leq d-1$.
        \For{$ \mu = 1, \dots, l $}
            \State Compute QR factorization of $ \mat{\mathbf{x}^{(\mu)}}{r_{\mu-1}, n_\mu}{ r_\mu} $, i.e.~$\mat{\mathbf{x}^{(\mu)}}{r_{\mu-1}, n_\mu}{ r_\mu} = Q \cdot R$ with $Q \in \R^{r_{\mu-1} \cdot n_\mu \times s}$ and $Q^T \cdot Q = I$.
            \State Define $\mathbf{y} \in \R^{r_{\mu-1} \times n_\mu \times s}$ as a reshaped version of $Q$ with $\mathbf{y}_{i,j,k } = Q_{\overline{i, j}, k}$.
            \State Define $\mathbf{z} \in \R^{s \times n_{\mu+1} \times r_{\mu+1} }$ by $\mat{\mathbf{z}}{s}{n_{\mu+1}, r_{\mu+1}} = R \cdot \mat{\mathbf{x}^{(\mu+1)}}{r_{\mu}}{n_{\mu+1}, r_{\mu+1}}$.
            \State Set $\mathbf{x}^{(\mu)}$  to $\mathbf{y}$, $\mathbf{x}^{(\mu+1)}$  to $\mathbf{z}$ and $r_{\mu}$ to $s$.
        \EndFor
    \end{algorithmic}
\end{algorithm}

\begin{algorithm}[htbp]
    \caption{Right-orthonormalize TT-cores $\mathbf{x}^{(l)}, \dotsc, \mathbf{x}^{(d)}$.}
    \label{alg:right-ortho}
    \begin{algorithmic}[1]
        \State Given tensor $ \mathbf{x} $ and core number $ l $, $2 \leq l \leq d$.
        \For{$ \mu = d, \dots, l $}
            \State Compute QR factorization of $\left( \mat{\mathbf{x}^{(\mu)}}{r_{\mu-1}}{n_\mu, r_\mu} \right)^T$, i.e.~$\mat{\mathbf{x}^{(\mu)}}{r_{\mu-1}}{n_\mu, r_\mu} = R^T \cdot Q^T$ with $Q^T \in \R^{s \times n_\mu \cdot r_\mu}$ and $Q^T \cdot Q = I$.
            \State Define $\mathbf{y} \in \R^{s \times n_\mu \times r_\mu}$ as a reshaped version of $Q^T$ with $\mathbf{y}_{i,j,k } = Q^T_{i, \overline{j, k}}$.
            \State Define $\mathbf{z} \in \R^{r_{\mu-2} \times n_{\mu-1} \times s }$ by $\mat{\mathbf{z}}{r_{\mu-2}, n_{\mu-1}}{s} = \mat{\mathbf{x}^{(\mu-1)}}{r_{\mu-2}, n_{\mu-1}}{r_{\mu-1}} \cdot R^T$.
            \State Set $\mathbf{x}^{(\mu)}$  to $\mathbf{y}$, $\mathbf{x}^{(\mu-1)}$  to $\mathbf{z}$ and $r_{\mu-1}$ to $s$.
        \EndFor
    \end{algorithmic}
\end{algorithm}

Algorithm~\ref{alg:left-ortho} shows the left-orthonormalization process. Similarly, the procedure for the right-orthonormalization is shown in Algorithm~\ref{alg:right-ortho}. If we apply the Algorithms~\ref{alg:left-ortho} or~\ref{alg:right-ortho} to a tensor $\mathbf{x}$, then the tensor itself remains the same, the algorithms simply compute a different but equivalent representation. It would also be possible to adapt both algorithms to use SVDs instead of QR factorizations. In this way, similar to Algorithm~\ref{alg:TT-decomposition}, it is possible to truncate the TT-cores during the orthonormalizations.

With the aid of the orthonormalization algorithms, we are now able to compute pseudoinverses of arbitrary tensor unfoldings. The idea is to left-orthonormalize all cores $\mathbf{x}^{(1)}, \dotsc, \mathbf{x}^{(l)}$ and right-orthonormalize all cores $\mathbf{x}^{(l+1)}, \dotsc, \mathbf{x}^{(d)}$ in order to determine the pseudoinverse of a matricization of a tensor train $\mathbf{x}$ with respect to the dimensions $(1, \dotsc, l)$ and $(l+1, \dotsc, d)$. At the same time, we keep the matrix containing the singular values corresponding to an SVD applied to the core $\mathbf{x}^{(l)}$ intact and regard it as a diagonal matrix between two vectorizations. In this way, we can construct a singular value decomposition of the whole tensor. 

The procedure for computing the pseudoinverse is described in Algorithm~\ref{alg:Pseudoinverse} and the steps of the algorithm are also illustrated in Figure~\ref{fig:Pseudoinverse}. Again, we make use of the same diagrammatic notation as in \cite{HRS12}. Each TT-core is depicted by a circle with 3 arms indicating the set of free indices and tensor coupling, i.e.~index contractions, are represented by joining corresponding arms. In order to visualize orthonormal tensor cores we draw half filled circles.

\begin{algorithm}[htbp]
    \caption{Compute the pseudoinverse of a matricization of a TT-tensor $\mathbf{x}$.}
    \label{alg:Pseudoinverse}
    \begin{algorithmic}[1]
    \State Given a tensor $ \mathbf{x} $ in TT-format and core number $1 \leq l \leq d-1$, compute pseudoinverse of $\mat{\mathbf{x}}{n_1, \dotsc, n_l}{ n_{l+1}, \dotsc, n_d}$.
    \State Left-orthonormalize $\mathbf{x}^{(1)}, \dotsc, \mathbf{x}^{(l-1)}$ and right-orthonormalize $\mathbf{x}^{(d)}, \dotsc, \mathbf{x}^{(l+1)}$ using Algorithms~\ref{alg:left-ortho} and~\ref{alg:right-ortho}.
    \State Compute SVD of $\mat{\mathbf{x}^{(l)}}{r_{l-1}, n_l}{r_l}$, i.e.~$\mat{\mathbf{x}^{(l)}}{r_{l-1}, n_l}{r_l} = U \Sigma V^T$ with $\Sigma \in \R^{s \times s}$.
    \State Define $\mathbf{y} \in \R^{r_{l-1} \times n_l \times s}$ as a reshaped version of $U$ with $\mathbf{y}_{i,j,k } = U_{\overline{i, j}, k}$.
    \State Define $\mathbf{z} \in \R^{s \times n_{l+1} \times r_{l+1} }$ by $\mat{\mathbf{z}}{s}{n_{l+1}, r_{l+1}} = V^T \cdot \mat{\mathbf{x}^{(l+1)}}{r_{l}}{n_{l+1}, r_{l+1}}$.
    \State Set $\mathbf{x}^{(l)}$  to $\mathbf{y}$, $\mathbf{x}^{(l+1)}$  to $\mathbf{z}$ and $r_l$ to $s$.
    \State Define $M = \mat{\left( \sum_{k_0 =1}^{r_0} \cdots \sum_{k_{l-1}=1}^{r_{l-1}} \mathbf{x}^{(1)}_{k_0, :, k_1 } \otimes \dotsc \otimes \mathbf{x}^{(l)}_{k_{l-1}, :, : } \right)}{n_1, \dotsc, n_l}{r_l}$.
    \State Define $N = \mat{\left( \sum_{k_{l+1} =1}^{r_{l+1}} \cdots \sum_{k_{d}=1}^{r_{d}} \mathbf{x}^{(l+1)}_{:, :, k_{l+1} } \otimes \dotsc \otimes \mathbf{x}^{(d)}_{k_{d-1}, :, k_d } \right)}{r_l}{n_{l+1}, \dotsc, n_d}$.
    \State Then $ \mat{\mathbf{x}}{n_1, \dotsc, n_l}{ n_{l+1}, \dotsc, n_d} = M \, \Sigma \, N $ and $ \left( \mat{\mathbf{x}}{n_1, \dotsc, n_l}{n_{l+1}, \dotsc, n_d} \right)^+ = N^T \, \Sigma^{-1} \, M^T$.
    \end{algorithmic}
\end{algorithm}

\begin{figure}[htb]
    \centering
    \includegraphics[width=0.7\textwidth]{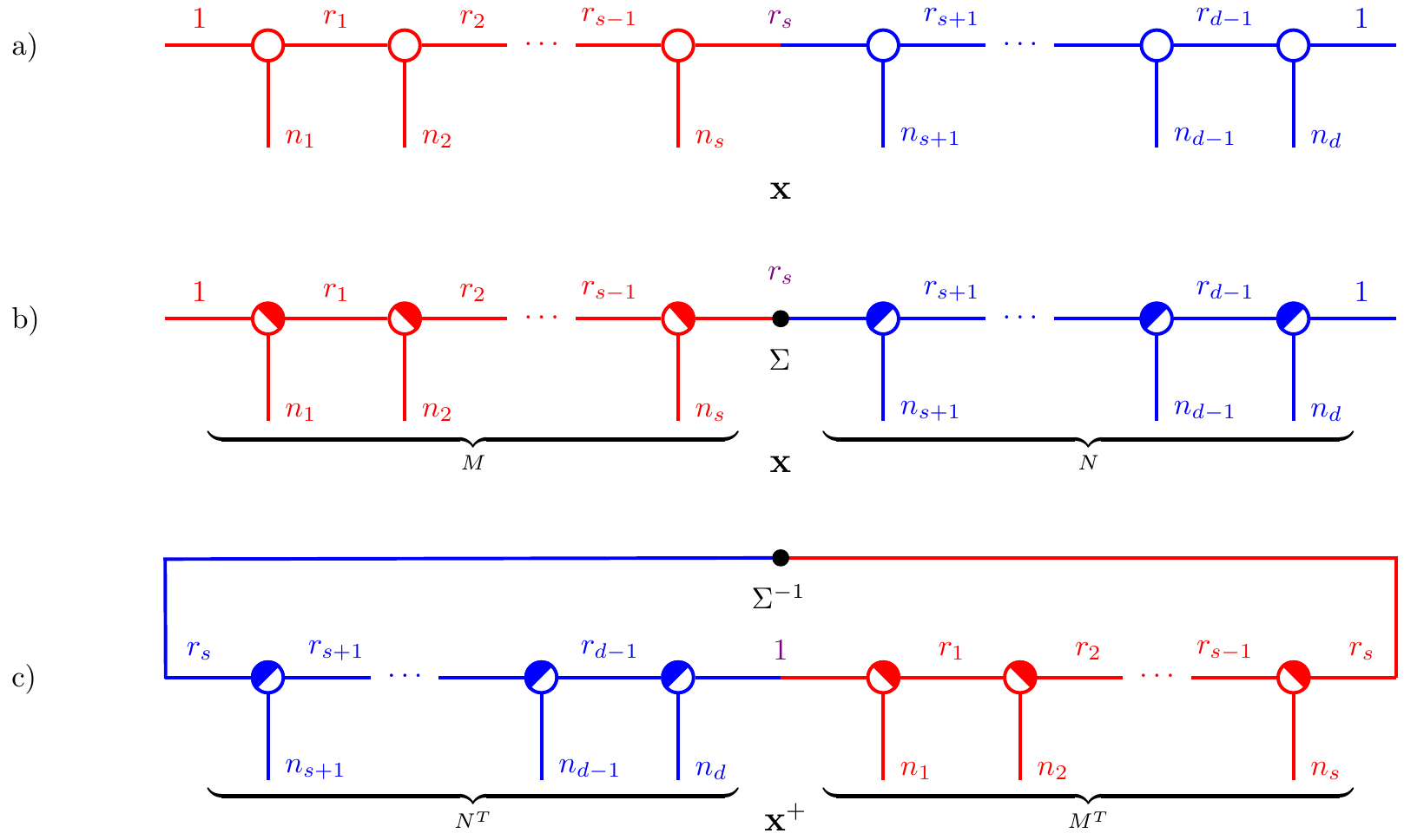}
    \caption{Illustration of the procedure to compute the pseudoinverse of a tensor in TT-format. a) Initial tensor $ \mathbf{x} $. b) Left- and right-orthonormalization of the tensor cores. c) Representation of the pseudoinverse $ \mathbf{x}^+ $.}
    \label{fig:Pseudoinverse}
\end{figure}

\begin{proposition}
Given a tensor $ \mathbf{x} $ and core number $1 \leq l \leq d-1$, Algorithm~\ref{alg:Pseudoinverse} computes the pseudoinverse with respect to the dimensions $ (1, \dots, l) $ and $ (l+1, \dots, d) $.
\end{proposition}
\begin{proof}
Since the left- and right-orthonormalization as well as the application of the SVD in line three of Algorithm~\ref{alg:Pseudoinverse} do not change the tensor $\mathbf{x}$ itself, we can express the matricization of $\mathbf{x}$ with respect to the dimensions $ (1, \dots, l) $ and $ (l+1, \dots, d) $ as
\begin{equation*}
    \mat{\mathbf{x}}{n_1, \dotsc, n_l}{ n_{l+1}, \dotsc, n_d} = M \, \Sigma \, N,
\end{equation*}
with $M $, $\Sigma$, and $N$ as given in Algorithm~\ref{alg:Pseudoinverse}. Now, we only have to show that $M^T \cdot M = N \cdot N^T = I \in \R^{r_l \times r_l}$. We obtain
\begin{equation*}
    \begin{split}
        M^T \cdot M &=  \left( \mat{\left( \sum_{k_0 =1}^{r_0} \cdots \sum_{k_{l-1}=1}^{r_{l-1}} \mathbf{x}^{(1)}_{k_0, :, k_1 } \otimes \dotsc \otimes \mathbf{x}^{(l)}_{k_{l-1}, :, : } \right)}{n_1, \dotsc, n_l}{r_l} \right)^T \\
        & \qquad \cdot \mat{\left( \sum_{k'_0 =1}^{r_0} \cdots \sum_{k'_{l-1}=1}^{r_{l-1}} \mathbf{x}^{(1)}_{k'_0, :, k'_1 } \otimes \dotsc \otimes \mathbf{x}^{(l)}_{k'_{l-1}, :, : } \right)}{n_1, \dotsc, n_l}{r_l}.
    \end{split}
\end{equation*}
Considering an entry of $M^T \cdot M$ and using Lemma~\ref{lem:Tensor product properties}, we then get
\begin{equation*}
    \begin{split}
        \left( M^T \cdot M \right)_{i,j} &=  \left( \mat{\left( \sum_{k_0 =1}^{r_0} \cdots \sum_{k_{l-1}=1}^{r_{l-1}} \mathbf{x}^{(1)}_{k_0, :, k_1 } \otimes \dotsc \otimes \mathbf{x}^{(l)}_{k_{l-1}, :, i } \right)}{n_1, \dotsc, n_l}{~} \right)^T \\
        & \qquad \cdot \mat{\left( \sum_{k'_0 =1}^{r_0} \cdots \sum_{k'_{l-1}=1}^{r_{l-1}} \mathbf{x}^{(1)}_{k'_0, :, k'_1 } \otimes \dotsc \otimes \mathbf{x}^{(l)}_{k'_{l-1}, :, j } \right)}{n_1, \dotsc, n_l}{~}\\
        &=  \sum_{k_0 =1}^{r_0} \cdots \sum_{k_{l-1}=1}^{r_{l-1}} \sum_{k'_0 =1}^{r_0} \cdots \sum_{k'_{l-1}=1}^{r_{l-1}} \prod_{\mu=1}^{l} \left( \mathbf{x}^{(\mu)}_{k_{\mu-1}, :, k_\mu } \right)^T \cdot \mathbf{x}^{(\mu)}_{k'_{\mu-1}, :, k'_\mu }, 
    \end{split}
\end{equation*}
with $k_l =i$ and $k'_l = j$. Since $\mathbf{x}^{(1)}$ is left-orthonormal and $r_0 =1$, we obtain $\left( \mathbf{x}^{(1)}_{1,:, k_1} \right)^T \cdot \mathbf{x}^{(1)}_{1,:, k'_1} = \delta_{k_1, k'_1}$. This implies that $\left( M^T \cdot M \right)_{i,j}$ is only nonzero if $k_1 = k'_1$. Now, we include the next core. This yields
\begin{equation*}
    \begin{split}
        & \sum_{k_1 = 1}^{r_1} \sum_{k'_1 = 1}^{r_1} \delta_{k_1, k'_1} \cdot \left( \mathbf{x}^{(2)}_{k_1,:, k_2} \right)^T \cdot \mathbf{x}^{(2)}_{k'_1,:, k'_2} \\ & \quad = \sum_{k_1 = 1}^{r_1} \left( \mathbf{x}^{(2)}_{k_1,:, k_2} \right)^T \cdot \mathbf{x}^{(2)}_{k_1,:, k'_2} = \left( \mat{\mathbf{x}^{(2)}_{:,:,k_2}}{r_1, n_1}{~}\right)^T \cdot \mat{\mathbf{x}^{(2)}_{:,:,k'_2}}{r_1, n_1}{~} = \delta_{k_2, k'_2}
    \end{split}
\end{equation*}
since $\mathbf{x}^{(2)}$ is also left-orthonormal. Successively, it then follows that, in order for $ \left( M^T \cdot M \right)_{i,j} $ to be nonzero, it must hold that $k_\mu = k'_\mu$ for $\mu = 2, \dotsc, l-1$. Thus, we obtain
\begin{equation*}
    \left( M^T \cdot M \right)_{i,j} = \sum_{k_{l-1}=1}^{r_{l-1}} \left( \mathbf{x}^{(l)}_{k_{l-1},:, i} \right)^T \cdot \mathbf{x}^{(l)}_{k_{l-1},:, j} = \left( \mat{\mathbf{x}^{(l)}_{:,:,i}}{r_{l-1}, n_l}{~}\right)^T \cdot \mat{\mathbf{x}^{(l)}_{:,:,j}}{r_{l-1}, n_l}{~}.
\end{equation*}
Note that due to the construction (see lines 3 \& 4 of Algorithm~\ref{alg:Pseudoinverse}) $\mathbf{x}^{(l)}$ is also left-orthonormal and therefore
\begin{equation*}
    M^T \cdot M = I \in \R^{r_l \times r_l}.
\end{equation*}
Analogously, it can be shown that $N \cdot N^T = I$ using the right-orthonormality of $\mathbf{x}^{(l+1)}, \dotsc, \mathbf{x}^{(d)}$. It follows that the pseudoinverse calculated by Algorithm~\ref{alg:Pseudoinverse} satisfies the necessary conditions, e.g.
\begin{align*}
        & \mat{\mathbf{x}}{n_1, \dotsc, n_l}{ n_{l+1}, \dotsc, n_d}  \cdot \left( \mat{\mathbf{x}}{n_1, \dotsc, n_l}{ n_{l+1}, \dotsc, n_d} \right)^+ \cdot \mat{\mathbf{x}}{n_1, \dotsc, n_l}{ n_{l+1}, \dotsc, n_d} \\
            & \quad = M \, \Sigma \, N \cdot N^T \, \Sigma^{-1} \, M^T \cdot M \, \Sigma \, N
            = M \, \Sigma \, N
            = \mat{\mathbf{x}}{n_1, \dotsc, n_l}{ n_{l+1}, \dotsc, n_d}. \qedhere
\end{align*}

\end{proof}

\begin{remark}
Algorithm~\ref{alg:Pseudoinverse} depicts just one possible way to compute the pseudoinverse of a given tensor. If the cores of the tensor train $\mathbf{x}$ are already left- or right-orthonormal, respectively, e.g.~the decomposition of $\mathbf{x}$ was computed by Algorithm~\ref{alg:TT-decomposition}, we can adapt Algorithm~\ref{alg:Pseudoinverse}. For instance, if all TT-cores (except the last one) are already left-orthonormal, we skip the application of Algorithm~\ref{alg:left-ortho} and only right-orthonormalize the cores $\mathbf{x}^{(d)}, \dotsc , \mathbf{x}^{(l+2)}$. Then the SVD can be applied to a matricization of $\mathbf{x}^{(l+1)}$ such that the matrix $V^T$ represents the updated version of $\mathbf{x}^{(l+1)}$ and $U$ is multiplied to the core $\mathbf{x}^{(l)}$ from the right. 
\end{remark}

Note that we do not need to compute $ M $ and $ N $ explicitly. Instead, we only execute the lines 1 to 6 of Algorithm~\ref{alg:Pseudoinverse} and then store the modified cores and the matrix $\Sigma$ containing the singular values. Thus, again with a slight abuse of the notation, we only store
\begin{equation*}
    \begin{split}
        \mathbf{x}^+ = \sum_{k_l = 1}^{r_l} \sigma^{-1}_{ k_l} & \cdot \left( \sum_{k_{l+1} =1}^{r_{l+1}} \cdots \sum_{k_{d-1}=1}^{r_{d-1}} \mathbf{x}^{(l+1)}_{k_l, :, k_{l+1} } \otimes \dotsc \otimes \mathbf{x}^{(d)}_{k_{d-1}, :, 1 }\right) \\
        & \cdot \left( \sum_{k_{1} =1}^{r_{1}} \cdots \sum_{k_{l-1}=1}^{r_{l-1}} \mathbf{x}^{(1)}_{1, :, k_{1} } \otimes \dotsc \otimes \mathbf{x}^{(l)}_{k_{l-1}, :, k_l }\right),
    \end{split}
\end{equation*}
which can be either regarded as the sum of $r_l$ tensor trains scaled by $\sigma^{-1}_{1}, \dotsc, \sigma^{-1}_{r_l} $ or as a cyclic tensor train as depicted in Figure \ref{fig:Pseudoinverse}. For a detailed description of the cyclic TT-format --~also called cyclic matrix product states~-- we refer to~\cite{Hac12}. Note that here again the cores $ 1, \dots, l $ and $ l+1, \dots, d $ are simply swapped as in Example~\ref{ex:rank-1 pseudoinverse}.

\subsection{Computation of DMD modes and eigenvalues in TT-format}

Let us now assume that the $ m $ snapshots are $ d $-dimensional arrays -- given for example by low-rank tensor representations -- of the form
\begin{equation}
     \mathbf{x}_i, \mathbf{y}_i \in \R^{n_1 \times \dots \times n_d},
\end{equation}
where $ \mathbf{y}_i = F(\mathbf{x}_i) $. These snapshots can be stored in the $(d+1)$-dimensional tensor trains $\mathbf{X}, \mathbf{Y} \in \R^{n_1 \times \dots \times n_d \times m} $, such that 

\begin{equation}
    \mathbf{X}_{:, \dotsc, :, i} = \mathbf{x}_i \quad \text{and} \quad \mathbf{Y}_{:, \dotsc, :, i} = \mathbf{y}_i,
\end{equation}
for $i=1, \dotsc, m$. Let $r_0,  \dotsc, r_{d+1}$ be the TT-ranks of $\mathbf{X}$ and $s_0,  \dotsc, s_{d+1}$ the TT-ranks of $\mathbf{Y}$. Now, let $X,Y \in \R^{n_1 \cdot \dotsc \cdot n_d \times m}$ be the specific matricizations of $\mathbf{X}$ and $\mathbf{Y}$, where we contract the dimensions $n_1, \dotsc, n_d$ such that every column of $X$ and $Y$, respectively, is the vectorization of the corresponding snapshot. We assume again that there is a linear relationship between the pairs of data vectors, i.e.
\begin{equation}
    Y = A X,
\end{equation}
with $A \in \R^{n_1 \cdot \dotsc \cdot n_d \times n_1 \cdot \dotsc \cdot n_d}$. We already stated in \eqref{eq:DMD A} that the linear operator $A$ can be computed by $Y \cdot X^+$. The pseudoinverse $X^+$ can be expressed -- after applying Algorithm~\ref{alg:Pseudoinverse} -- as
\begin{equation} \label{eq:representation of X^+}
    X^+ = \left( \mat{\mathbf{X}}{n_1, \dotsc, n_d }{m} \right)^+ = N^T \, \Sigma^{-1} \, M^T.
\end{equation}
Using similar matricizations, we can also represent the tensor unfolding $Y$ as a matrix product, i.e.
\begin{equation} \label{eq:representation of Y}
    \begin{split}
        Y &= \mat{\mathbf{Y}}{n_1, \dotsc, n_d }{m} \\
           &= \mat{\left(\sum_{l_{0} =1}^{s_{0}} \cdots \sum_{l_{d-1}=1}^{s_{d-1}} \mathbf{Y}^{(1)}_{l_0, :, l_{1} } \otimes \dotsc \otimes \mathbf{Y}^{(d)}_{l_{d-1}, :, : }\right) }{n_1, \dotsc, n_d}{s_{d}} \cdot \mat{\mathbf{Y}^{(d+1)}}{s_{d}}{m}
           = P \, Q.
    \end{split}
\end{equation}
Note that we do not require any special properties of the tensor cores of $\mathbf{Y}$. Left- and right-orthonormality must only hold for the TT-cores of $\mathbf{X}$. Combining the representations of $X^+$ and $Y$, we can express the matrix $ A $ as 
\begin{equation} \label{eq:full matrix in TT}
    A = Y \cdot X^+ =  P \, Q  \cdot  N^ T  \, \Sigma^{-1} \, M^T .
\end{equation}
As explained in Section~\ref{sec:Dynamic mode decomposition}, there are different algorithms to compute the eigenvalues and eigenvectors of $ A $. Instead of computing $ A $ explicitly, we are interested in the reduced matrix $\tilde{A} \in \R^{r_d \times r_d}$ as mentioned in Algorithm~\ref{alg:StandardDMD}. Rewriting the expression for $\tilde{A}$ using the decompositions given in \eqref{eq:representation of X^+} and \eqref{eq:representation of Y}, this results in
\begin{equation} \label{eq:reduced matrix in TT}
    \tilde{A} = M^T \cdot P \, Q \cdot N^ T  \, \Sigma^{-1}.
 \end{equation}
In order to compute $\tilde{A}$, we do not have to compute the matrices $M^T$ and $P$ explicitly. We bypass this computational cost by splitting \eqref{eq:reduced matrix in TT} into different parts. First, consider $M^T \cdot P$, any entry is given by
\begin{equation*}
    \begin{split}
        \left( M^T \cdot P \right)_{i,j} &= \left( \mat{\left( \sum_{k_0 =1}^{r_0} \cdots \sum_{k_{d-1}=1}^{r_{d-1}} \mathbf{X}^{(1)}_{k_0, :, k_1 } \otimes \dotsc \otimes \mathbf{X}^{(d)}_{k_{d-1}, :, i } \right)}{n_1, \dotsc, n_d}{~} \right)^T\\
        & \qquad \cdot  \mat{\left(\sum_{l_{0} =1}^{s_{0}} \cdots \sum_{l_{d-1}=1}^{s_{d-1}} \mathbf{Y}^{(1)}_{l_0, :, l_{1} } \otimes \dotsc \otimes \mathbf{Y}^{(d)}_{l_{d-1}, :, j }\right) }{n_1, \dotsc, n_d}{~}.
    \end{split}
\end{equation*}
It follows from the linearity of matricizations and Lemma~\ref{lem:Tensor product properties} that
\begin{equation*}
    \left( M^T \cdot P \right)_{i,j} = \sum_{\mathclap{k_0 =1}}^{r_0} \cdots \sum_{k_{d-1}=1}^{r_{d-1}} \sum_{l_{0}=1}^{s_{0}} \cdots \sum_{\mathclap{l_{d-1}=1}}^{s_{d-1}} \left( \mathbf{X}^{(1)}_{k_0, :, k_1 } \right)^T \mathbf{Y}^ {(1)}_{l_0, :, l_1} \cdot \dotsc \cdot \left( \mathbf{X}^{(d)}_{k_{d-1}, :, i } \right)^T \mathbf{Y}^ {(d)}_{l_{d-1}, :, j}.
\end{equation*}
Thus, by defining the matrices $\Theta_{\mu} \in \R^{r_{\mu-1} \cdot s_{\mu-1} \times r_\mu \cdot s_\mu}$ as
\begin{equation} \label{eq:Theta 1}
    \left( \Theta_{\mu} \right)_{\overline{k_{\mu-1}, l_{\mu-1}}, \overline{k_\mu, l_\mu}} = \left( \mathbf{X}^{(\mu)}_{k_{\mu-1}, :, k_\mu } \right)^T \cdot \mathbf{Y}^ {(\mu)}_{l_{\mu-1}, :, l_\mu},
\end{equation} 
for $\mu = 1, \dotsc, d$, we can write any entry of $M^T \cdot P$ as
\begin{equation} \label{eq:Theta 2}
    \left( M^T \cdot P \right)^T_{i,j} = \Theta_1 \cdot \Theta_2 \cdot \dotsc \cdot \Theta_{d-1} \cdot \left( \Theta_d \right)_{:, \overline{i,j}}.
\end{equation}
In this way, we can compute $M^T \cdot P$ without leaving the TT-format, we only have to reshape certain contractions of the TT-cores as depicted in \eqref{eq:Theta 1} and \eqref{eq:Theta 2}. This computation can be implemented efficiently using Algorithm 4 from \cite{Ose11}. The result is then a low-dimensional matrix with $r_d$ rows and $s_d$ columns, assuming that the TT-ranks of $\mathbf{X}$ and $\mathbf{Y}$ are small compared to the whole state space of these tensors. Indeed, the tensor ranks $r_d$ and $s_d$ are both bounded by the number of snapshots $m$ due to the right-orthonormalization of the last TT-cores. For the second term of~\eqref{eq:reduced matrix in TT}, $Q \cdot N^T$, we simply obtain
\begin{equation*}
    Q \cdot N^ T = \mat{\mathbf{Y}^{(d+1)}}{s_{d}}{m} \cdot \left( \mat{\mathbf{X}^{(d+1)}}{r_{d}}{m} \right)^T.
\end{equation*}
Subsequently, we only have to multiply the three low-dimensional matrices $\left(M^T \cdot P \right)$, $\left( Q \cdot N^T \right)$ and $\Sigma^{-1}$. The latter is just a diagonal matrix containing the reciprocals of the singular values occurring in $\Sigma$. Overall, we do not need to convert any tensor products of cores of $\mathbf{X}$ or $\mathbf{Y}$, respectively, into full tensors during our calculations. The results are naturally low-dimensional matrices and the reduced matrix $\tilde{A}$ can finally be used to compute the eigenvalues of the high-dimensional matrix $ A $ since both have the same spectrum. 

In order to calculate the corresponding DMD modes of $ A $, consider again Algorithms~\ref{alg:StandardDMD} and \ref{alg:ExactDMD}. If $ \lambda_1, \dotsc, \lambda_\nu $ are the eigenvalues of $ \tilde{A} $ corresponding to the eigenvectors $ w_1, \dotsc, w_\nu \in \R^{r_d} $, then the DMD modes $ \varphi_1, \dotsc, \varphi_\nu \in \R^{n_1 \cdot \dotsc \cdot n_d} $ of $ A $ according to the standard algorithm are given by 
\begin{equation} \label{eq:ModesStandard}
    \varphi_\mu = M \cdot w_\mu, 
\end{equation}
for $\mu=1, \dotsc, \nu$. At this point, we again benefit from using the TT-representations of $ M $. What \eqref{eq:ModesStandard} tells us is actually just the replacement of the last TT-core. This can be seen by defining $W \in \R^{r_d \times \nu}$ as
\begin{equation}
    W =
    \begin{bmatrix}
        w_1 & w_2 & \cdots & w_\nu
    \end{bmatrix}
\end{equation}
and writing
\begin{equation} \label{eq:DMDmodesStandard}
    M \cdot W = \mat{\left( \sum_{k_0 =1}^{r_0} \cdots \sum_{k_{d-1}=1}^{r_{d-1}} \mathbf{X}^{(1)}_{k_0, :, k_1 } \otimes \dotsc \otimes \mathbf{X}^{(d)}_{k_{d-1}, :, : } \right)}{n_1, \dotsc, n_d}{r_{d}} \cdot W.
\end{equation}
As a result, we can also express the DMD modes in a TT-representation, i.e.
\begin{equation}
\Phi = \sum_{k_0 =1}^{r_0} \cdots \sum_{k_{d}=1}^{r_{d}} \mathbf{X}^{(1)}_{k_0, :, k_1 } \otimes \dotsc \otimes \mathbf{X}^{(d)}_{k_{d-1}, :, k_d } \otimes W_{k_d, :},
\end{equation}
with $\mat{\Phi_{:, \dotsc, :, \mu}}{n_1, \dotsc, n_d}{~} = \varphi_\mu$ for $\mu = 1, \dotsc, \nu$. Considering the exact DMD algorithm, the DMD modes are given by
\begin{equation} \label{eq:ModesExact}
    \varphi_\mu = \frac{1}{\lambda} \cdot P \, Q \cdot N^T \, \Sigma^{-1} \cdot w_\mu, 
\end{equation}
for $\mu = 1, \dotsc, \nu$. The tensor train $\Phi$ representing all DMD modes is then given by
\begin{equation}
    \Phi = \sum_{k_0 =1}^{s_0} \cdots \sum_{k_{d}=1}^{s_{d}} \mathbf{Y}^{(1)}_{k_0, :, k_1 } \otimes \dotsc \otimes \mathbf{Y}^{(d)}_{k_{d-1}, :, k_d } \otimes {\underbrace{\left(Q \cdot N^T \, \Sigma^{-1} \cdot W \cdot \Lambda^{-1} \right) }_{\in \R^{s_d \times \nu}}} \,^{\phantom{(d)}}_{k_d, :},
\end{equation}
again with $\mat{\Phi_{:, \dotsc, :, \mu}}{n_1, \dotsc, n_d}{~} = \varphi_\mu$ for $\mu = 1, \dotsc, \nu$ and 
\begin{equation}
    \Lambda =
    \begin{pmatrix}
        \lambda_1 & & 0\\
        & \ddots & \\
        0 & & \lambda_\nu
    \end{pmatrix}.
\end{equation}
Summing up, we can express the DMD modes using a previously given tensor train, modifying just the last core. For the standard DMD algorithm, we can express $\Phi$ using the first $d$ cores of $\mathbf{X}$ and replacing the core $\mathbf{X}^{(d+1)}$ by $W$. As for the exact DMD algorithm, $\Phi$ is represented by the first $d$ cores of $\mathbf{Y}$, replacing $\mathbf{Y}^{(d+1)}$ by $Q \, N^T \, \Sigma^{-1} \, W \, \Lambda^{-1} $. In both cases, we benefit from not leaving the TT-representations of $ \mathbf{X} $ and $\mathbf{Y}$, respectively.

\section{Numerical results}
\label{sec:Numerical results}

The following examples are mainly for illustration purposes, we will not describe the underlying mathematical models in full detail since the governing equations are not relevant here. Instead of analyzing simulation data, we could also process experimental measurement data. The goal is to detect the dominant dynamics of a dynamical system given only data. The first example, a simulation of two merging vortices, is two-dimensional and has been created with the \textsc{Multimod} toolbox~\cite{Rou13}. The second example is a three-dimensional simulation of the flow around a blunt body governed by the incompressible Navier--Stokes equations and has a significantly higher number of degrees of freedom. 

The DMD experiments using the TT-format were performed on a Linux machine with 128 GB RAM and an Intel Xeon processor with a clock speed of 3 GHz and 8 cores. The algorithms were implemented in MATLAB R2015a using a compound of cell arrays and multidimensional matrices for tensors in the TT-format.

\paragraph{Two merging vortices.}

The first example shows two merging vortices. Here, the domain is discretized using an $ n \times n $ grid with $ n = 200, 400, \dots, 1400 $. We generated data for $ 447 $ equidistant time steps, thus~$ \mathbf{X}, \mathbf{Y} \in \R^{n \times n \times 446} $. Intermediate solutions are shown in Figure~\ref{fig:Vortex}a, the corresponding DMD modes in Figure~\ref{fig:Vortex}b. We computed tensor representations of $X$ and $Y$ using Algorithm~\ref{alg:TT-decomposition} in order to compare the results for different thresholds $\varepsilon \geq 0$. Nevertheless, we assume that the data matrices $ \mathbf{X}, \mathbf{Y} $ are already given in the TT-format for the tensor-based formulation (the runtimes of Algorithm~\ref{alg:TT-decomposition} are thus not included in the overall runtimes), i.e.~the partial differential equation is directly solved using tensor representations. For example, that could mean that the numerical solution of the partial differential equation is obtained by applying an appropriate time-stepping scheme combined with the Alternating Linear Scheme (ALS)~\cite{HRS12}. In this case, we can execute the half-sweeps of the ALS in a way that the resulting tensor is already left-orthonormal. Thus, all relevant calculations to compute the DMD modes are included in the runtimes presented in Figure~\ref{fig:VortexRuntimes}. We compare the exact DMD implementations, i.e.~Algorithm~\ref{alg:ExactDMD} and its tensor-based counterpart.

\begin{figure}[htbp]
    \centering
    \subfiguretitle{a)}
    \includegraphics[width=0.95\textwidth]{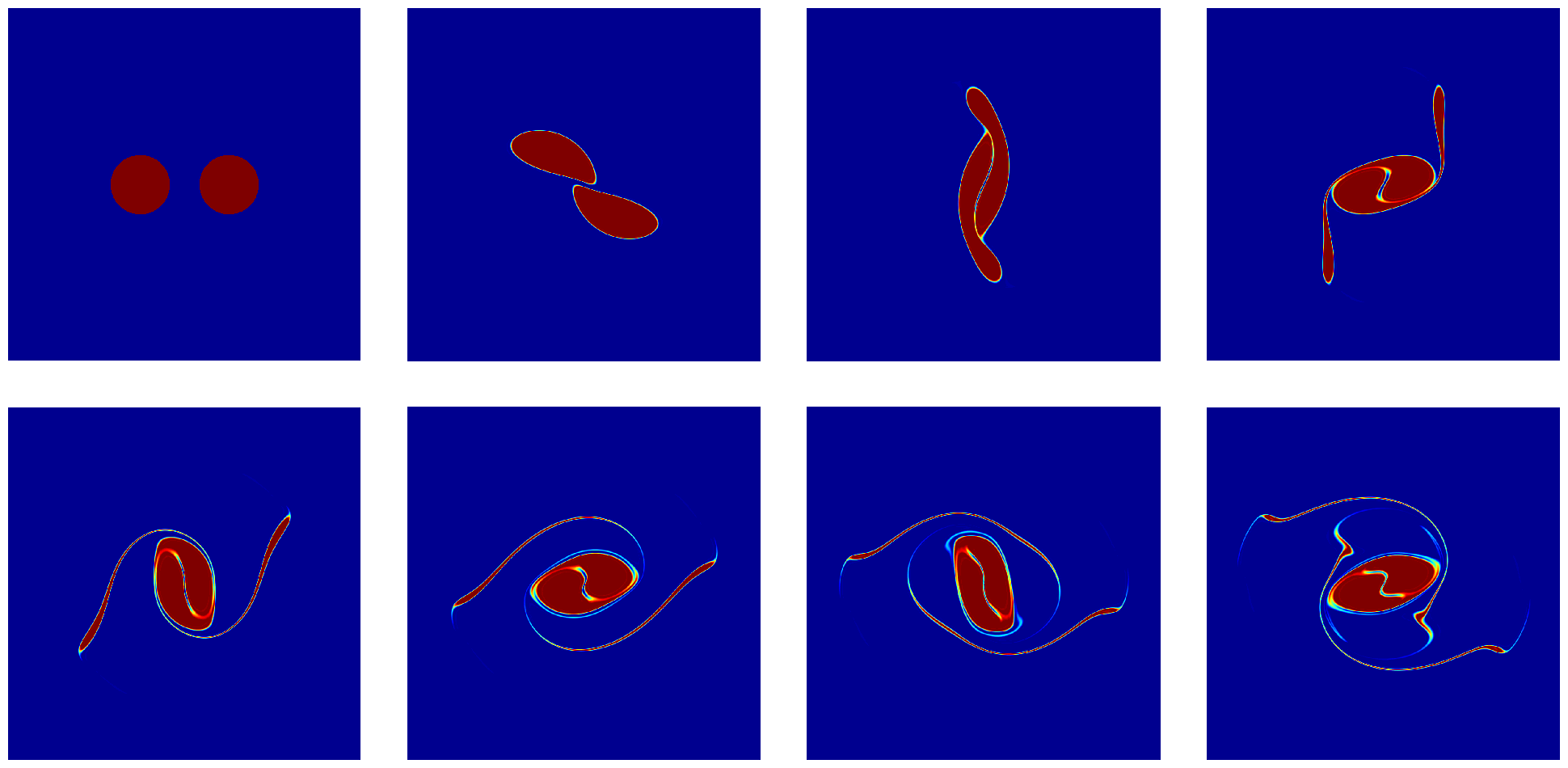} \\
    \subfiguretitle{b)}
    \includegraphics[width=0.95\textwidth]{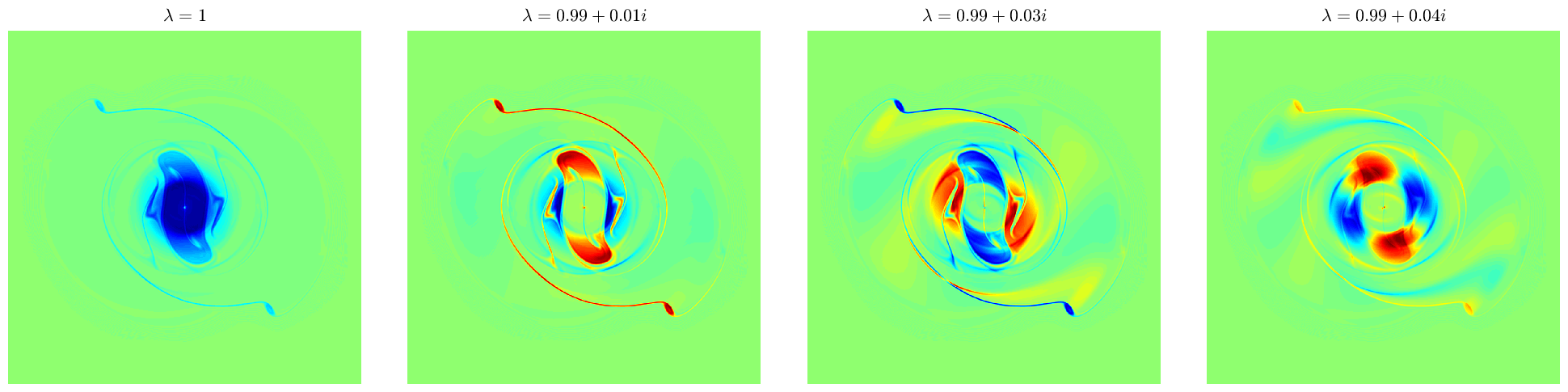}
    \caption{a) Simulation of two merging vortices. b) DMD modes corresponding to eigenvalues close to 1.}
    \label{fig:Vortex}
\end{figure}

The DMD modes shown in Figure~\ref{fig:Vortex}b are visually indistinguishable for the different thresholds $ \varepsilon $ and the resulting tensor ranks. The influence of the low-rank approximation on the numerical errors is shown in Table~\ref{tab:VortexErrors}, which contains the relative errors defined by $ e_\lambda = \lvert \lambda - \tilde{\lambda} \rvert /\abs{\lambda} $ and $ e_\varphi = \norm{\varphi - \tilde{\varphi}}_F/\norm{\varphi}_F $, where $ \lambda $ and $ \varphi $ are the DMD eigenvalue and mode for $ \varepsilon = 0 $ and $ \tilde{\lambda} $ and $ \tilde{\varphi} $ the corresponding approximations for a given $ \varepsilon > 0 $. Here, we normalized each mode in such a way that the largest absolute value is $ 1 $. For $ n = 1400 $, the speedup is approximately $ 7 $ for $ \varepsilon = 0 $ and $ 16 $ for $ \varepsilon = 1\mathrm{e}{-10} $. The ranks of the corresponding tensor trains $ \mathbf{X} $ are $ r = [1, 1400, 446, 1] $ and $ r = [1, 732, 446, 1] $, respectively. That is, approximately half of the singular values between the first and second core are less than $ \varepsilon = 1\mathrm{e}{-10} $. This illustrates that a given tensor representation of the data can be exploited to efficiently compute the DMD modes and eigenvalues without converting the data set to the full format since the tensor-train format already contains information about the required pseudoinverse. Furthermore, the results show that the lower the rank of the tensor approximation -- which depends on the parameter $ \varepsilon $ -- the higher the speedup.

\begin{figure}[htbp]
    \centering
    \includegraphics[width=0.8\textwidth]{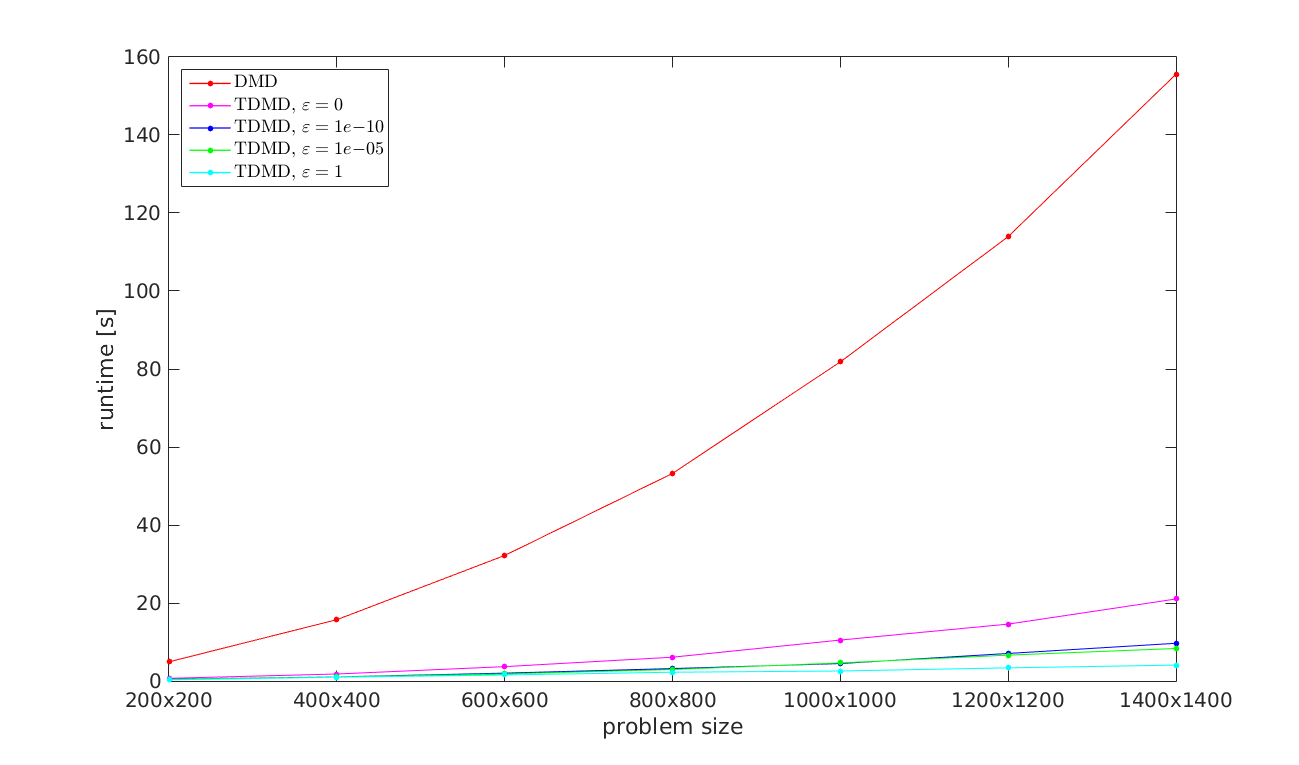}
    \caption{Comparison of the runtimes of DMD and tensor-based DMD applied to the vortex merging example for different values of $ \varepsilon $ and different problem sizes.}
    \label{fig:VortexRuntimes}
\end{figure}

\begin{table}
    \centering
    \caption{Influence of the truncation on the accuracy of the leading DMD modes for the $ 1400 \times 1400 $ grid discretization of the vortex merging example. The entries show the relative errors $ e_\lambda $ and $ e_\varphi $ for different values of $ \varepsilon $.}
    \label{tab:VortexErrors}
    \newcommand{\md}[2]{\multicolumn{#1}{c|}{#2}}
    \newcommand{\me}[2]{\multicolumn{#1}{c}{#2}}
    \scalebox{0.75}{
    \begin{tabular}{|l|rr|rr|rr|rr|}
        \hline
                                        & \md{2}{Mode 1} & \md{2}{Mode 2} & \md{2}{Mode 3} & \md{2}{Mode 4} \\
                                        & \me{1}{$ e_\lambda $} & \md{1}{$ e_\varphi $} & \me{1}{$ e_\lambda $} & \md{1}{$ e_\varphi $} & \me{1}{$ e_\lambda $} & \md{1}{$ e_\varphi $} & \me{1}{$ e_\lambda $} & \md{1}{$ e_\varphi $} \\ \hline
        $\varepsilon=1\mathrm{e}{-10} $ & 5.55e$-$15 & 5.92e$-$14 & 5.31e$-$15 & 3.32e$-$13 & 9.75e$-$15 & 3.50e$-$13 & 1.37e$-$14 & 3.28e$-$13 \\
        $\varepsilon=1\mathrm{e}{-05} $ & 1.47e$-$14 & 2.23e$-$10 & 5.09e$-$15 & 4.27e$-$10 & 4.93e$-$15 & 9.73e$-$10 & 1.39e$-$14 & 7.77e$-$10 \\
        $\varepsilon=1                $ & 6.10e$-$08 & 5.22e$-$04 & 2.90e$-$06 & 3.16e$-$03 & 3.02e$-$06 & 2.62e$-$03 & 4.00e$-$06 & 7.86e$-$04 \\ \hline
    \end{tabular}}
\end{table}

\paragraph{Flow around blunt body.}

As a second example, we consider the flow around a blunt body governed by the three-dimensional incompressible Navier-Stokes equations. Similar experiments have been described in~\cite{RMBSH09}. Here, the domain $\Omega \subset \mathbb{R}^3$ has a size of $L = (25, 15, 10)$. A conical object is placed inside the domain with the center axis at $(x_1, x_2) = (5, 7.5)$ and a diameter of $D_1 = 0.8$ at the boundaries and $D_2 = 1.6$ in the middle of the channel (cf.~Figure~\ref{fig:Cylinder}). We set the Reynolds number $Re = \overline{U} D/\nu = 240$, based on the inflow velocity $ \overline{U} = (1, 0, 0) $, the kinematic viscosity $ \nu = 5\cdot10^{-3} $, and the average cone diameter $ D = 1.2 $. Moreover, we apply periodic boundary conditions in the $ x_2 $ as well as the $ x_3 $ direction. The domain is discretized by a rectangular grid with approximately $ 10^6 $ degrees of freedom and the computations are performed using OpenFOAM \cite{JJT07}. Before applying the tensor-based DMD algorithm, $ 1001 $ snapshots are interpolated on an equidistant, rectangular grid of dimension $ 150 \times 85 \times 80 $, where the values for all grid points inside the object are set to zero.

For the DMD analysis, we consider the velocity magnitude $ \mathbf{U} = \sqrt{U_1^2 + U_2^2 + U_3^2} $, thus~$ \mathbf{X}, \mathbf{Y} \in \R^{150 \times 85 \times 80 \times 1000} $. The resulting ranks, runtimes, and relative errors for different values of $ \varepsilon $ are shown in Table~\ref{tab:BluntBodyData}. By increasing $ \varepsilon $, the initially high ranks can be reduced without a huge loss of accuracy. Here, we compared only the first two eigenmodes shown in Figure~\ref{fig:Cylinder_DMD} with the DMD modes for $ \varepsilon = 0 $. The runtime of conventional DMD for this problem is approximately 125\,s. The results show that the efficiency of DMD can be improved significantly using tensor decompositions, provided that the ranks are reasonably low. The aim is thus to directly compute low-rank solutions for such problems.

\begin{table}
    \centering
    \caption{Influence of the truncation on the ranks, runtimes, and accuracy of the leading DMD modes for the blunt body problem.}
    \label{tab:BluntBodyData}
    \newcommand{\md}[2]{\multicolumn{#1}{c|}{#2}}
    \newcommand{\me}[2]{\multicolumn{#1}{c}{#2}}
    \scalebox{0.75}{
    \begin{tabular}{|l|r|r|rr|rr|}
        \hline
                                        &         \md{1}{Ranks}              & \md{1}{Runtime} & \md{2}{Mode 1} & \md{2}{Mode 2}  \\
                                        &                                    &                 & \me{1}{$ e_\lambda $} & \md{1}{$ e_\varphi $} & \me{1}{$ e_\lambda $} & \md{1}{$ e_\varphi $} \\ \hline
        $\varepsilon=0    $ & [1,   150,  6083,  1000,               1] & 134\,s & 0          & 0          & 0          & 0          \\
        $\varepsilon=0.01 $ & [1,   150,  4708,  1000,               1] & 102\,s & 6.29e$-$07 & 6.33e$-$04 & 1.53e$-$07 & 2.88e$-$04 \\
        $\varepsilon=0.05 $ & [1,   150,  3649,  \phantom{0}641,     1] &  52\,s & 1.11e$-$04 & 3.64e$-$02 & 3.05e$-$05 & 2.37e$-$02 \\
        $\varepsilon=0.1  $ & [1,   148,  3003,  \phantom{0}527,     1] &  35\,s & 1.38e$-$04 & 6.92e$-$02 & 1.26e$-$04 & 3.65e$-$02 \\
        $\varepsilon=0.5  $ & [1,   135,  1624,  \phantom{0}343,     1] &  14\,s & 2.62e$-$04 & 8.69e$-$02 & 9.61e$-$05 & 5.55e$-$02 \\
        $\varepsilon=1    $ & [1,   130,  1199,  \phantom{0}278,     1] &   8\,s & 4.34e$-$04 & 1.72e$-$01 & 1.80e$-$04 & 1.06e$-$01 \\ \hline
    \end{tabular}}
\end{table}

\begin{figure}[htb]
    \centering
    \includegraphics[width=0.95\textwidth]{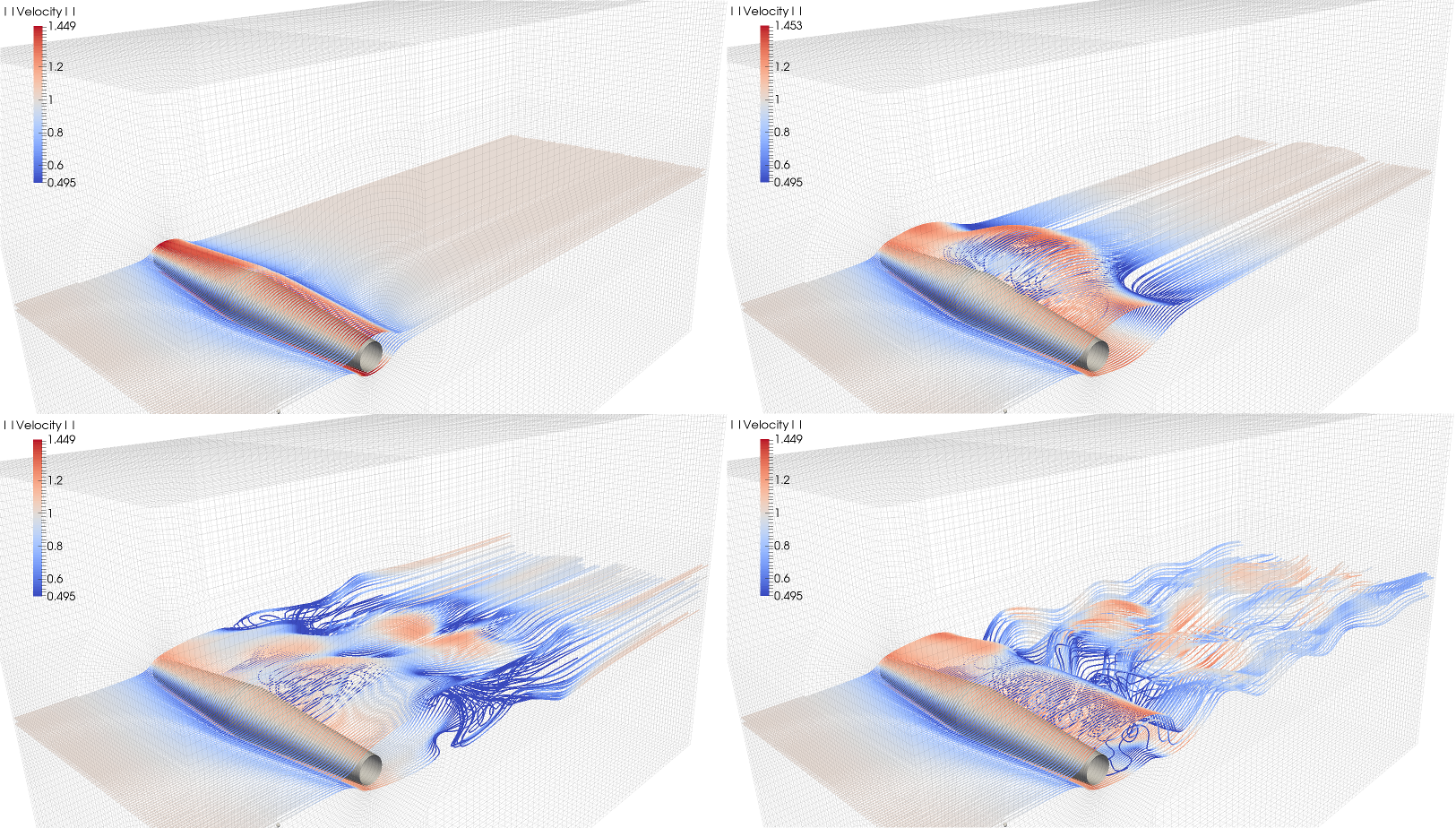}
    \caption{Simulation of the flow around a blunt body. The flow is visualized by streamlines which are inserted at the inflow and slightly above and below the cone axis, respectively. The streamlines are colored according to the velocity magnitude.}
    \label{fig:Cylinder}
\end{figure}

\begin{figure}[p]
    \centering
    \subfiguretitle{a) $\lambda = 0.998 + 0.049i$; $\omega = 0.4906$ \hspace{4.5cm} b) $\lambda = 0.997 + 0.075i$; $\omega = 0.7508$}
    \includegraphics[width=0.475\textwidth]{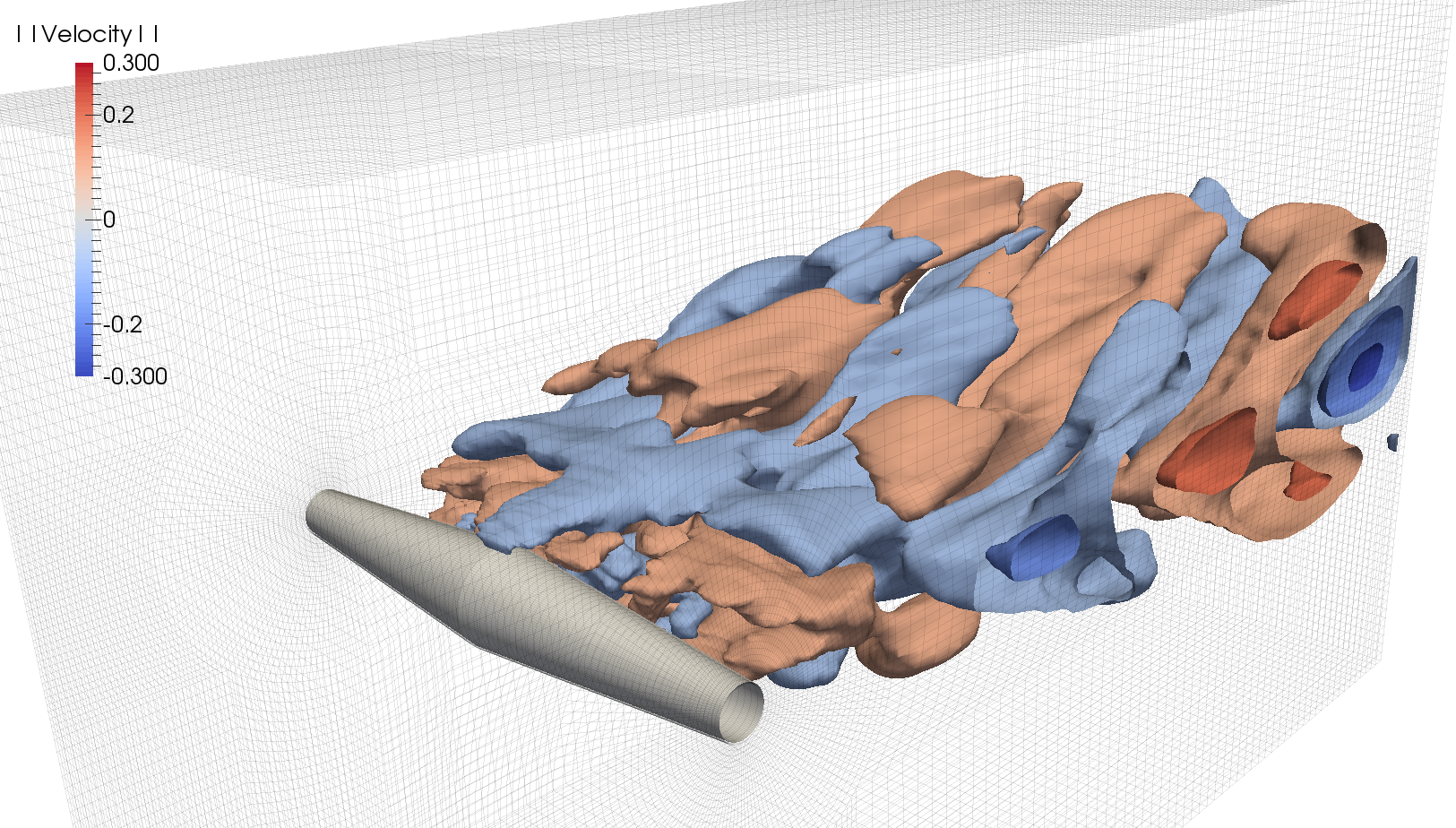}
    \includegraphics[width=0.475\textwidth]{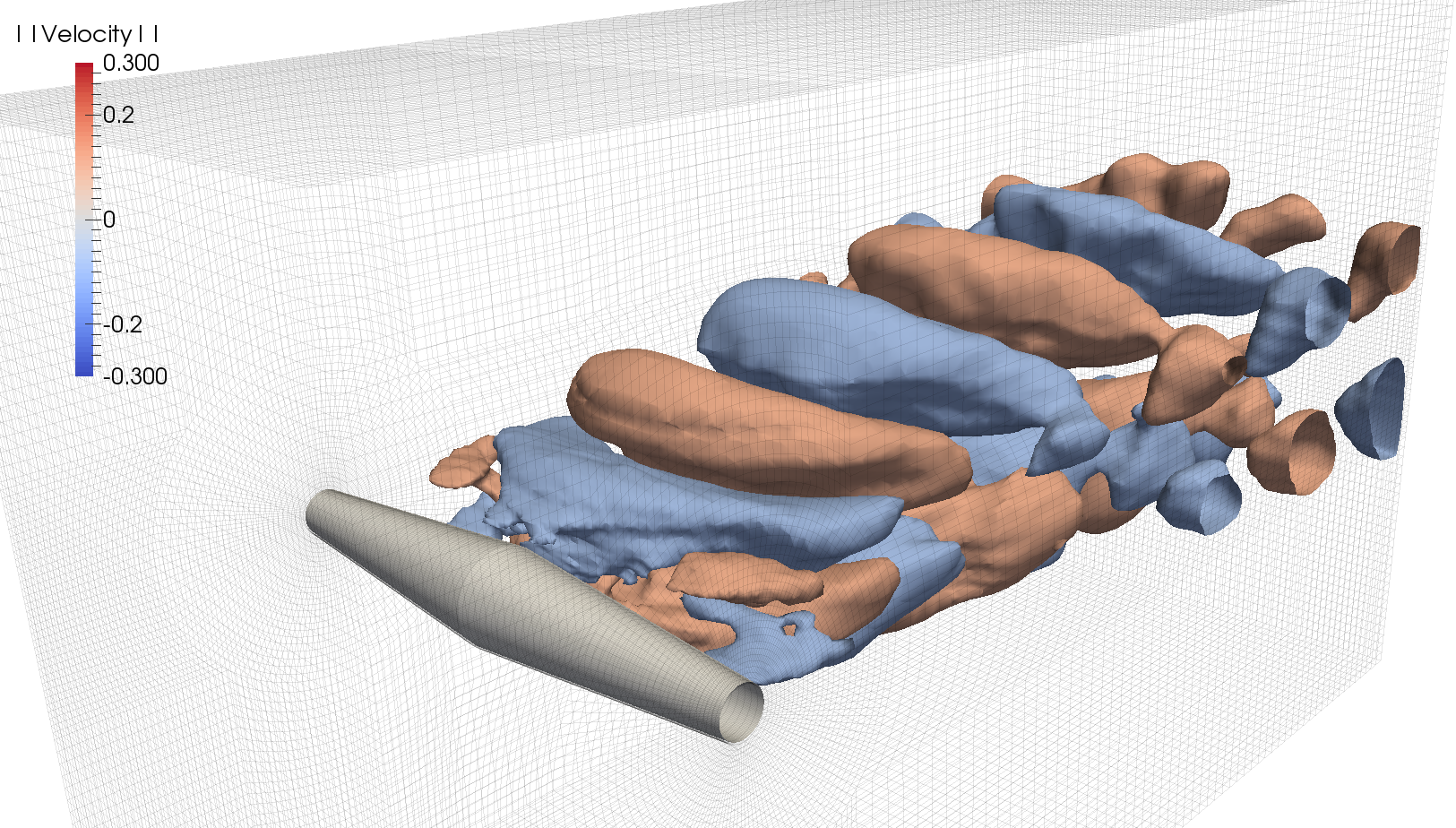}\\
    \subfiguretitle{c) $\lambda = 0.992 + 0.117i$; $\omega = 1.1740$ \hspace{4.5cm} d) $\lambda = 0.992 + 0.123i$; $\omega = 1.2336$}
    \includegraphics[width=0.475\textwidth]{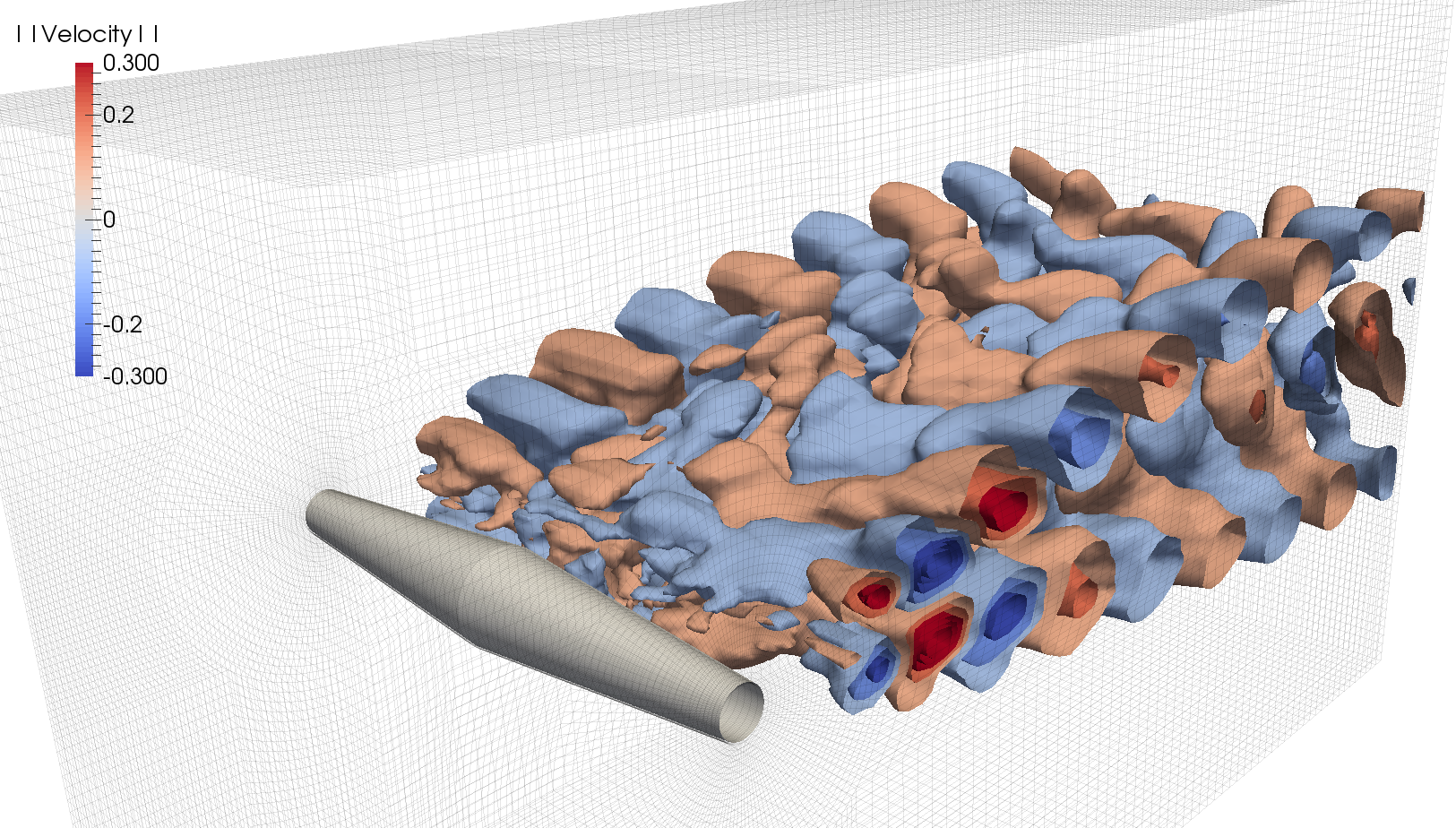}
    \includegraphics[width=0.475\textwidth]{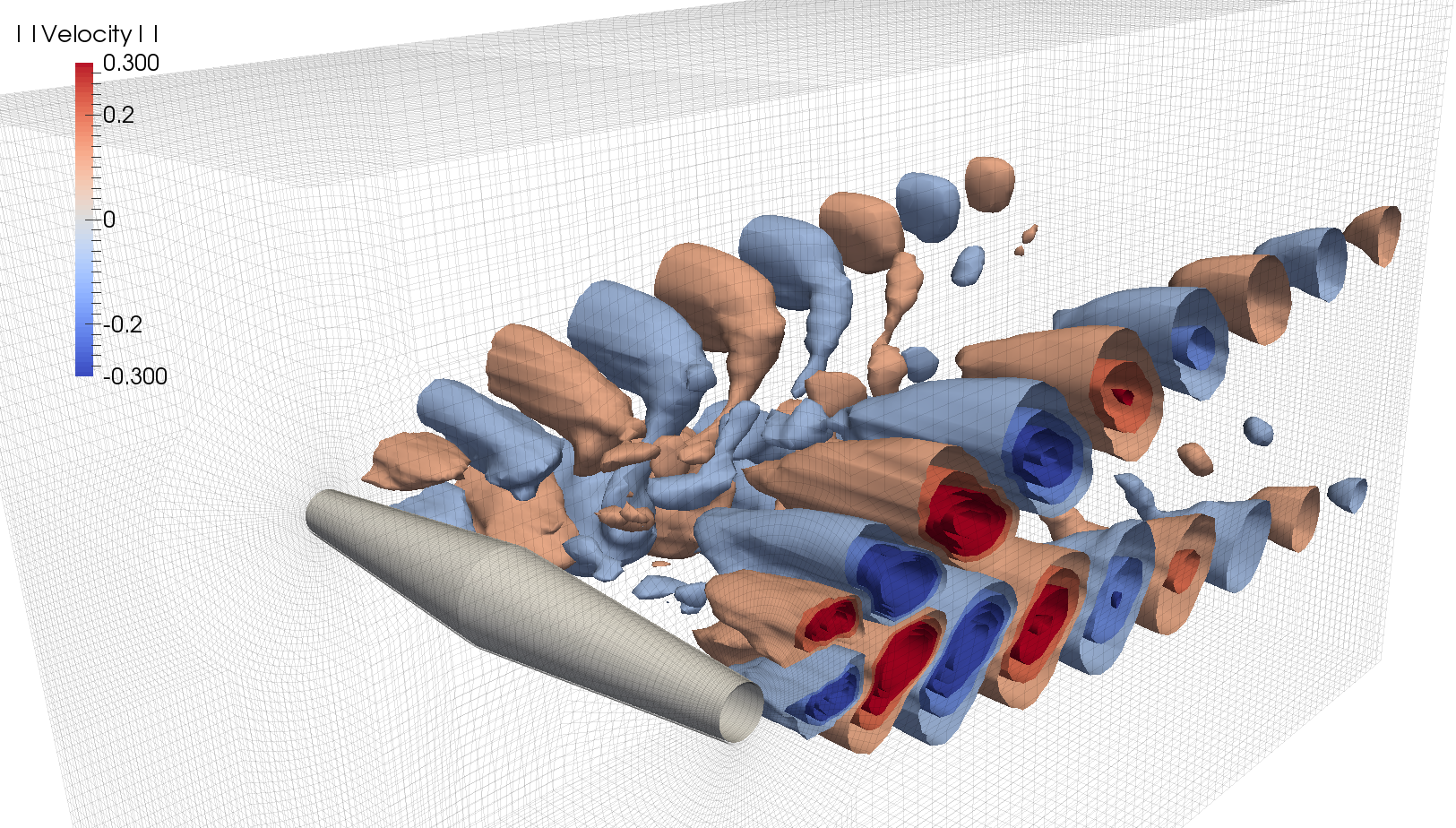}\\
    \subfiguretitle{e) $\lambda = 0.991 + 0.131i$; $\omega = 1.3143$ \hspace{4.5cm} f) $\lambda = 0.979 + 0.197i$; $\omega = 1.9857$}
    \includegraphics[width=0.475\textwidth]{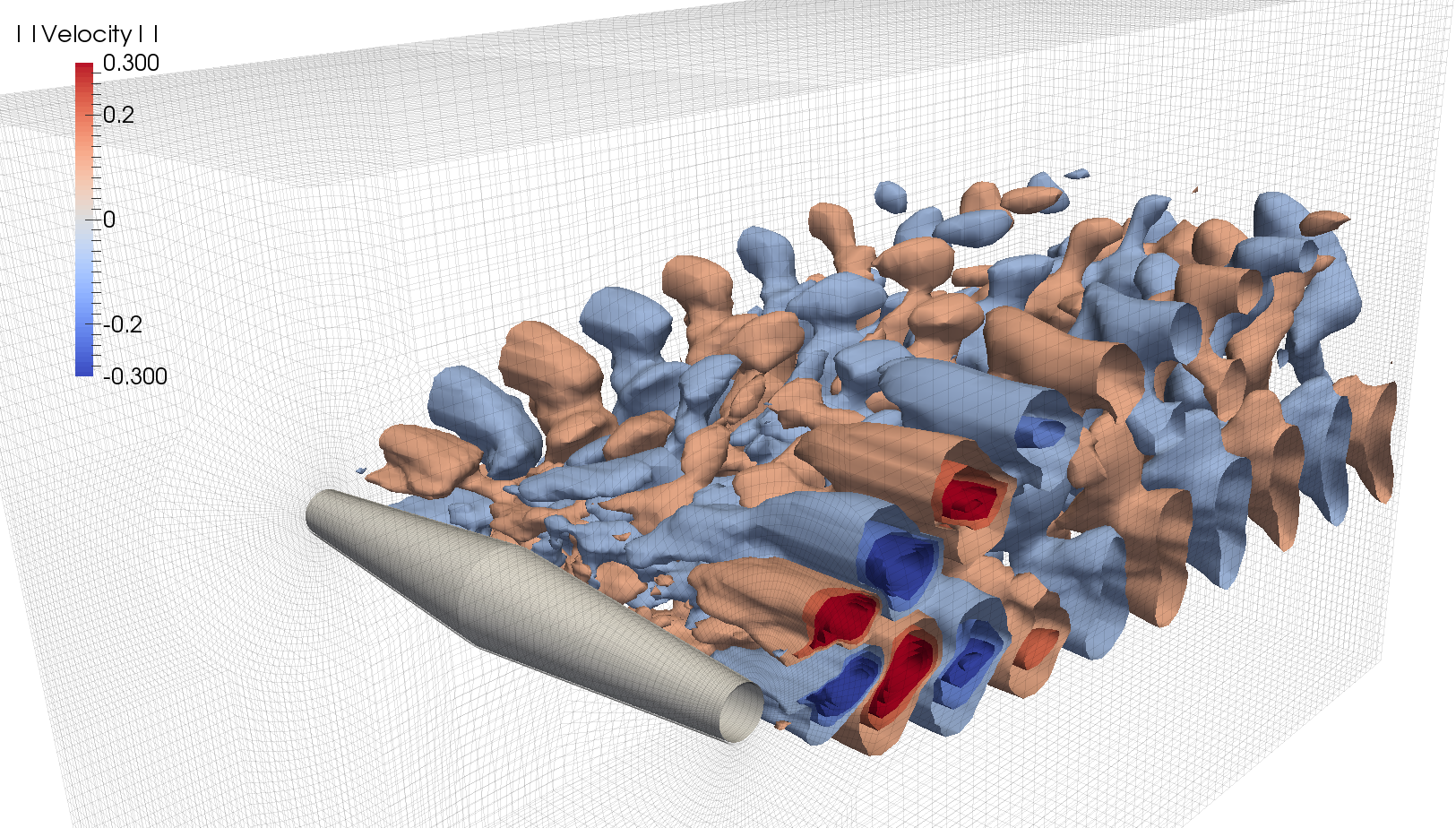}
    \includegraphics[width=0.475\textwidth]{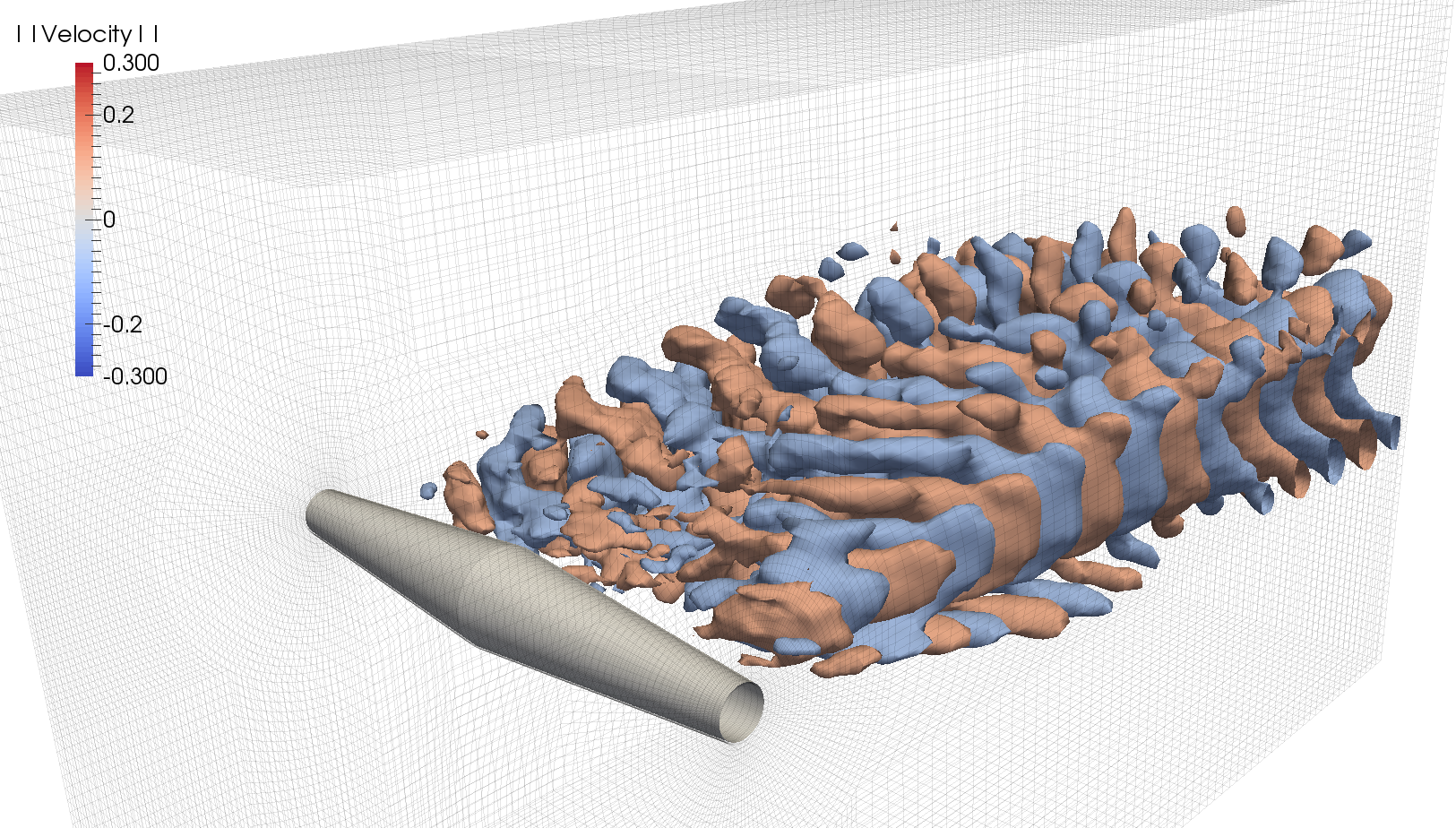}\\
    \subfiguretitle{g) $\lambda = 0.969 + 0.244i$; $\omega = 2.4668$ \hspace{4.5cm} h) $\lambda = 0.929 + 0.369i$; $\omega = 3.7809$}
    \includegraphics[width=0.475\textwidth]{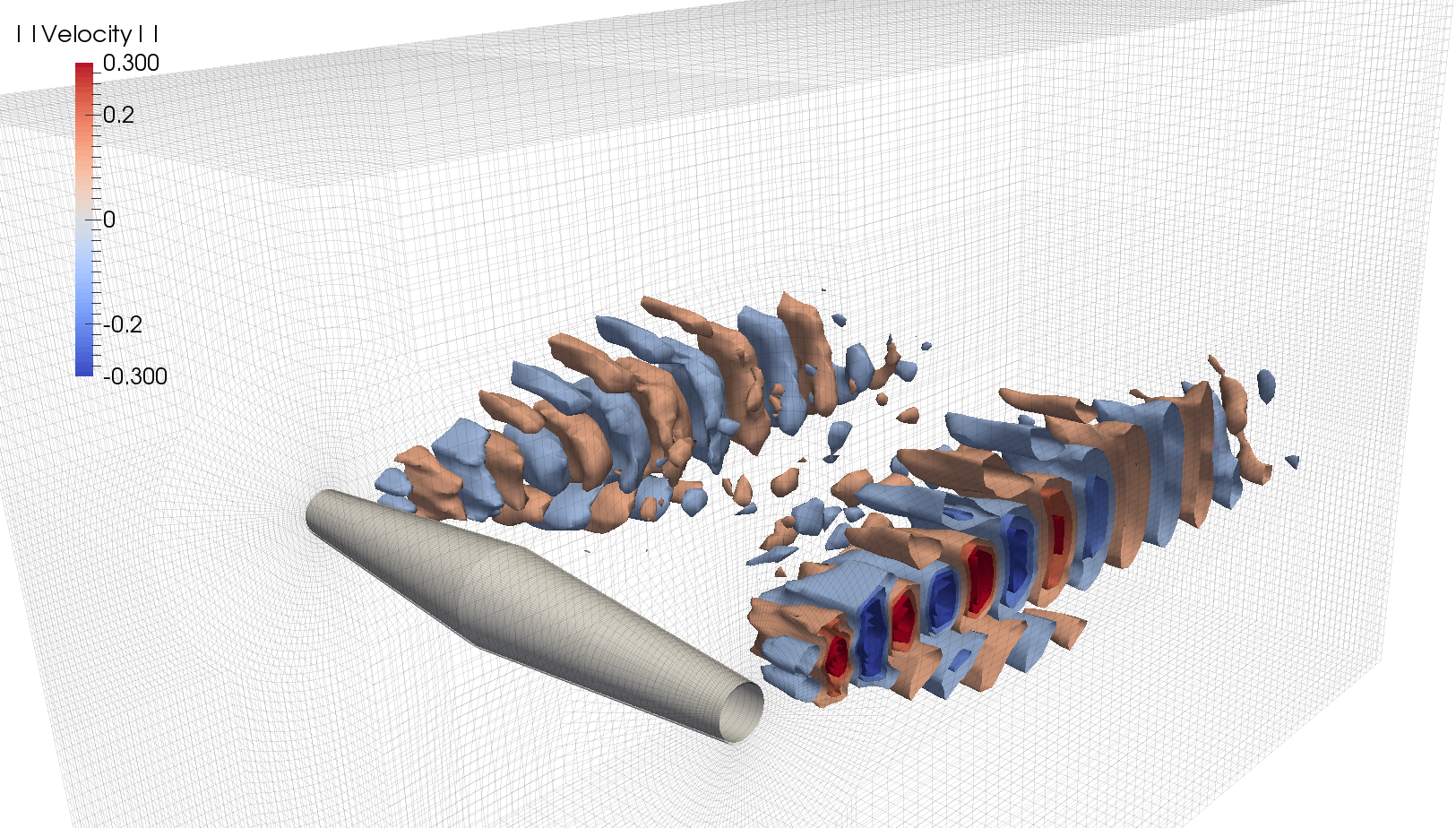}
    \includegraphics[width=0.475\textwidth]{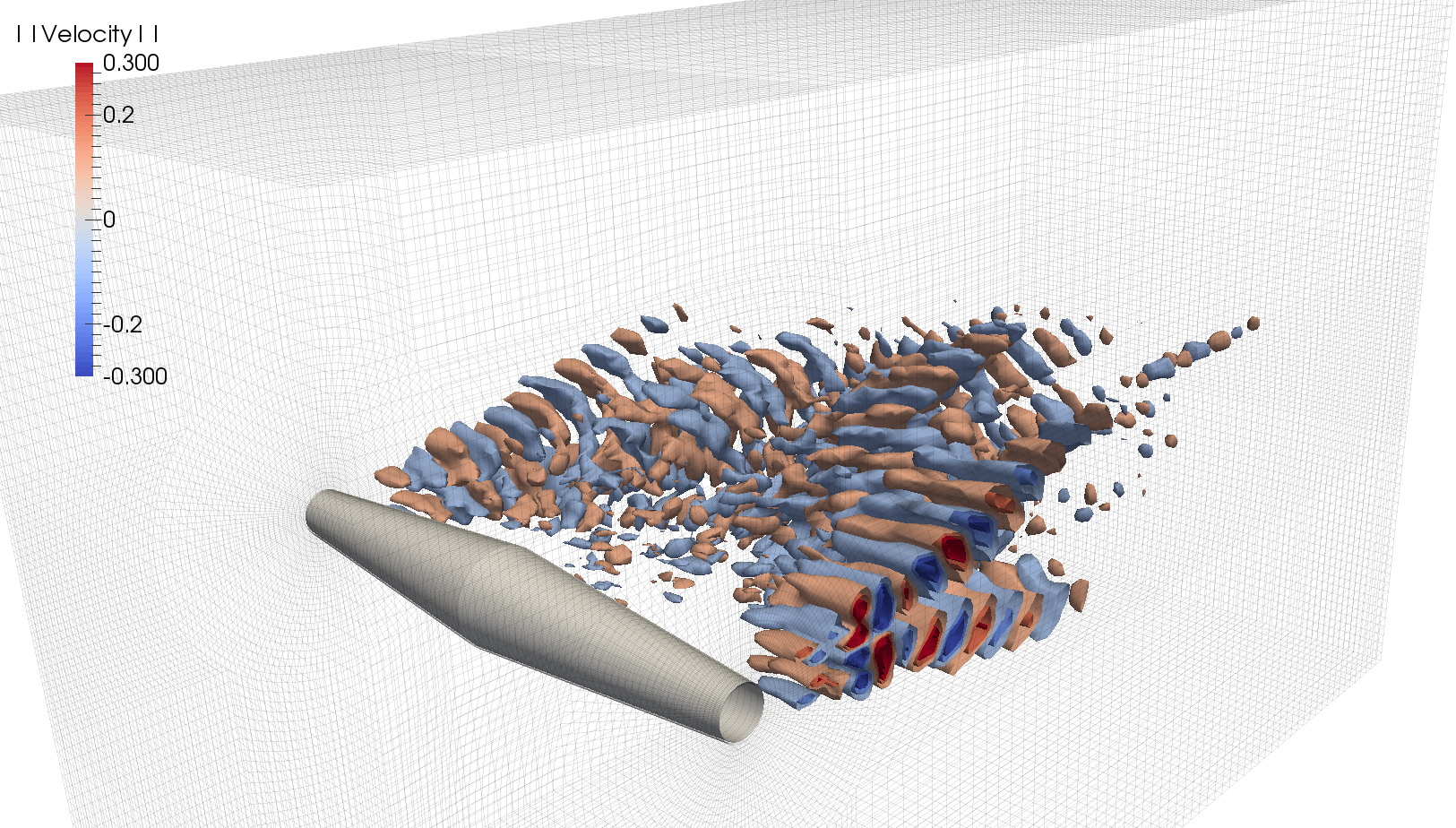}
    \caption{DMD modes for the flow around a cone corresponding to eigenvalues close to 1. The modes are visualized by iso-surfaces of the velocity magnitude.}
    \label{fig:Cylinder_DMD}
\end{figure}

Figure~\ref{fig:Cylinder} shows streamlines of the velocity field at different time steps. Similar to the two-dimensional von K\'arm\'an vortex street (Figure~\ref{fig:Karman}), vortices separate alternatingly from the lower and the upper side of the blunt body. However, due to the conical shape of the object, the flow pattern is much more complex than in the two-dimensional case. In this case, storing the full matrix $ A $ would require more than $7.5$ TB. The corresponding DMD modes, sorted by their respective frequency, are depicted in Figure~\ref{fig:Cylinder_DMD}, where we visualize the flow patterns by iso-surfaces of the velocity magnitude. In accordance with \cite{RMBSH09}, we calculate the frequency of the modes by $\omega = \Im(\log(\lambda))/\Delta t$, where $\Delta t$ is the time step between the snapshots. We observe that larger structures correspond to lower frequencies, i.e.~they indicate large but slowly rotating vortices, whereas the smaller structures possess higher frequencies. Moreover, we can see that large structures originate from the outward pointing kink in the middle of the domain and the smaller structures from the inward pointing kink at the periodic boundary in $x_2$-direction.

\section{Conclusion}
\label{sec:Conclusion}

We showed that the TT-format -- which is based on successive SVDs -- implicitly already contains information about the pseudoinverse of certain tensor unfoldings. The goal is to gain insight into the characteristic properties of tensors and to develop tensor-based algorithms for solving, for instance, systems of linear equations or eigenvalue problems that directly exploit the inherent properties of the TT-decomposition instead of reformulating the problem as an equivalent optimization problem. One application which requires the computation of a pseudoinverse of such a tensor unfolding is DMD. If the data to be analyzed is already given in TT-format, our algorithm efficiently computes the DMD modes and eigenvalues directly on the low-rank representations of the data matrices $ \mathbf{X} $ and $ \mathbf{Y} $.

Analogously, variants of DMD such as sparsity-promoting DMD or kernel-based DMD could be reformulated as well. In the same way, the algorithms presented within this paper could be used to extend EDMD to compute the Koopman modes directly in the TT-format, provided that the simulation data is generated using low-rank tensor approximations. The only difference is that we then have to compute the matrix $ A = \Psi_Y \Psi_X^+ $, where $ \Psi_X $ and $ \Psi_Y $ are nonlinear transformations of the original data matrices $ X $ and $ Y $. A tensor-based method to compute the eigenfunctions of the Koopman operator has been proposed in~\cite{KS16}. The computations rely on the construction of a generalized eigenvalue problem of the form $ \boldsymbol{\xi} \mathbf{A} = \lambda \boldsymbol{\xi} \mathbf{G} $, which is solved using simple power iteration schemes.

So far, we considered several two-dimensional problems and one three-dimensional problem to illustrate how tensor-based data-driven methods might help mitigate the curse of dimensionality and enable the analysis of more complex dynamical systems. Future work includes applying tensor-based DMD to higher-dimensional problems to analyze scalability and efficiency of the proposed methods as well as studying fluid flow applications with a larger number of degrees of freedom.

\section*{Acknowledgements}

This research has been partially funded by Deutsche For\-schungsgemeinschaft (DFG) through grant CRC 1114 as well as by the Berlin Mathematical School and the Einstein Center for Mathematics. The OpenFOAM calculations were performed on resources provided by the Paderborn Center for Parallel Computing ($PC^2$). Moreover, we would like to thank the reviewers for their helpful comments and suggestions for improvement.

\bibliographystyle{unsrt}
\bibliography{TDMD}

\end{document}